\documentclass[10pt,a4paper,reqno]{amsart}

\usepackage{graphicx}
\usepackage{mathscinet}
\usepackage[backend=biber,sorting=nty,natbib=true,isbn=false,doi=false,url=false,maxbibnames=5,style=ieee]{biblatex}
\usepackage[english]{babel}
\usepackage{csquotes}

\usepackage{microtype} 
\usepackage[title,titletoc,page]{appendix}
\usepackage[a4paper,left=3cm,right=3cm,top=2cm,bottom=4cm,bindingoffset=5mm]{geometry}
\usepackage[colorlinks=true,linkcolor=blue,linktocpage=true]{hyperref}

\usepackage{amsfonts}
\usepackage{amsmath}
\usepackage{amssymb}
\usepackage{dsfont}
\usepackage{mathtools}
\mathtoolsset{showonlyrefs}

\usepackage{thmtools}
\usepackage[shortlabels]{enumitem}

\usepackage{version}
\usepackage[obeyFinal]{todonotes}
\setuptodonotes{inline}
\usepackage{cancel}
\usepackage[normalem]{ulem}

\numberwithin{equation}{section}
\usepackage{bbm}
\usepackage{bm}
\usepackage[normalem]{ulem}
\usepackage{cancel}

\newcommand{\cA}{\mathcal{A}}
\newcommand{\cB}{\mathcal{B}}
\newcommand{\cC}{\mathcal{C}}
\newcommand{\cD}{\mathcal{D}}

\newcommand{\cF}{\mathcal{F}}

\newcommand{\cK}{\mathcal{K}}
\newcommand{\cL}{\mathcal{L}}

\newcommand{\cN}{\mathcal{N}}

\newcommand{\cP}{\mathcal{P}}

\newcommand{\cU}{\mathcal{U}}

\newcommand{\bE}{\mathbb{E}}

\newcommand{\bP}{\mathbb{P}}

\newcommand{\bR}{\mathbb{R}}

\newcommand{\dd}{\, \mathrm{d}}
\def\eps{\varepsilon}

\newcommand{\lipn}[1]{\left\| #1 \right\|_{\mathrm{Lip}}}

\newcommand{\ip}[2]{\! \left\langle #1,#2 \right\rangle \!}

\newcommand{\UCb}{\mathrm{UC}_{b}}
\newcommand{\LUC}{\mathrm{LUC}_b}

\newcommand{\Hbar}{\overline{H}}

\newcommand{\Hmin}{H_{\mathrm{min}}}
\newcommand{\Hcv}{H_{\mathrm{CV}}}

\newtheorem{theorem}{Theorem}[section]
\newtheorem{lemma}[theorem]{Lemma}

\newtheorem{proposition}[theorem]{Proposition}

\theoremstyle{definition}
\newtheorem{definition}[theorem]{Definition}
\newtheorem{remark}[theorem]{Remark}

\newtheorem{hypothesis}[theorem]{Hypothesis}

\title[PEL and Stationary HJB Equations]{Projected Evolutionary Lifting and Well-Posedness of Stationary Hamilton--Jacobi--Bellman Equations in Infinite Dimensions}
\author{Gabriele Bolli}
\address{Dipartimento di Matematica Guido Castelnuovo, Universit\`a degli Studi di Roma La Sapienza, Roma, Italy}
\email{gabriele.bolli@uniroma1.it}

\author{Fabian Fuchs}
\address{Department of AI, Data and Decision Sciences, Luiss Guido Carli, Roma, Italy}
\email{ffuchs@luiss.it}
\date{\today}

\thanks{The author are supported by the Italian Ministry of University and Research
(MIUR) within the framework of PRIN project 20223PNJ8K}

\begin{document}

%%%%%%%%%%%%%%%%%%%%%%%%%%%%%%%%%%%%%%%%%%%%
\begin{abstract}
This paper establishes the existence and uniqueness of mild solutions to stationary Hamilton-Jacobi-Bellman (HJB) equations associated with infinite-horizon stochastic optimal control problems in separable Hilbert spaces. Our framework includes settings with a lack of global smoothing properties of the transition semigroup, singular dynamics involving unbounded control operators, and state-dependent running costs. We overcome these challenges by lifting the state space using the Projected Evolutionary Lifting technique. This work is an extension of \cite{BoGo25}, in which existence and uniqueness is proved via a contraction mapping argument and is consequently restricted to sufficiently large discount factors. We remove this restriction, proving existence and uniqueness for any discount rate $\lambda > 0$ using tools from the theory of maximally monotone operators.\smallskip \ 
\end{abstract}

\maketitle
%%%%%%%%%%%%%%%%%%%%%%%%%%%%%%%%%%%%%%%%%%%%

\section{Introduction}

This paper deals with a class of stationary Hamilton--Jacobi--Bellman (HJB) equations arising from infinite time horizon stochastic optimal control problems in separable Hilbert spaces. The primary challenges in the investigation of such problems stem from three fundamental structural features: 
First, the lack of smoothing properties of the transition semigroup associated with the free evolution dynamics requires us to work with partial smoothing results. To attain the partial smoothing results we need, in turn, to lift the state in a suitable fashion.
Moreover, this lifting allows us to treat running costs that are state-dependent.
Finally, the presence of singular dynamics, arising, for instance, from boundary control problems, which involve unbounded control and/or diffusion operators require working in a control extension of the underlying Hilbert space. 
The present work is an extension of \cite{BoGo25}, which addresses the same problem by considering its mild form and finding a unique fixed point under the assumption of a sufficiently large discount factor $\lambda > 0$. We extend that result, establishing the existence and uniqueness of mild solutions for any discount rate $\lambda > 0$.

The complexities of our framework prevent the direct application of standard results available in the literature based on fixed point theorems on spaces of continuous functions or Gauss--Sobolev spaces as developed, for instance, in \cite{CadP91, CadP92, DPZa02, Go96} and \cite{ChMe97, GoGo06}, respectively. 
Viscosity solution theory, see, e.g., \cite[Chapter 3]{FaGoSw17}, can be used to prove existence and uniqueness of solutions. However, such solutions are usually merely continuous, which is a significant limitation in this context, as one needs at least differentiability of the solution in the space variable to prove the existence of optimal feedback control strategies through the dynamic programming approach. 
Lastly, not even BSDE-based techniques, see, e.g., \cite{FuTe02, FuTe04, FuTe2002} or \cite[Chapter 6]{FaGoSw17}, can be applied in our context due to the lack of the so-called `structure condition', which requires the control operator to act at most in the directions where the noise acts. This condition is violated in all examples we consider later. 

Consequently, more sophisticated techniques based on partial smoothing properties of the transition semigroup have been introduced. The seminal work \cite{GoMa17} investigates directional regularization properties for specific classes of functions, leading to the resolution of the finite-horizon problem in the absence of global smoothing; see also \cite{MaTe22,bodefe26} for similar results in this direction. This partial smoothing framework was subsequently extended by \cite{GoMa23} to include singular dynamics with unbounded control operators. Further refinements were achieved in \cite{GoMa25}, where the authors treat state-dependent costs by lifting the state space to a suitable space of trajectories and showing that a partial smoothing property holds in such a space. 
In \cite{BoGo25}, the existence and uniqueness of mild solutions for stationary HJB equations is proved using a contraction mapping principle, exploiting the fact that, for sufficiently large $\lambda > 0$, the integral operator associated with the equation is a contraction. While significant, this result is notably restrictive from the perspective of economic and financial applications, where the discount factor is typically an exogenously given parameter, such as an interest rate, rather than a variable depending on other model data. This work bridges this gap by extending the results of \cite{BoGo25} to any discount factor $\lambda > 0$.

The results in this paper rely on the theory of maximally monotone operators on Banach spaces, cf.\ \cite{barbu2010}, and are inspired by the techniques presented in \cite[Section 4.6.2.2]{FaGoSw17}. 
However, the complex structural assumptions of our framework, like the singular dynamics, unbounded control operators, and lack of global smoothing, prevent a straightforward application of those classical results.

These strict requirements motivate a deeper investigation of the concept of what we call \textit{Projected Evolutionary Lifting} (PEL), a technique initially introduced in \cite{GoMa25} and subsequently expanded in \cite{BoGo25}.\ The core intuition behind the PEL is to bypass the lack of regularity in the original state space by lifting the problem to a suitable space of trajectories.\ In this lifted framework, relevant functions do not depend on the state statically, but rather through the entire evolutionary trajectory generated by the free dynamics and filtered by a suitable projection. While the transition semigroup of the underlying problem may fail to smooth functions globally, one can show a partial smoothing effect along specific directions using this lifting.

Accordingly, the contributions of this paper are mainly twofold and linked to this methodology: On one hand, we extend the existing existence and uniqueness results to arbitrary discount factors $\lambda > 0$, removing this restrictive assumption found in prior literature. On the other hand, to adapt the classical monotone operator techniques to our singular setting, we refine the methodology of the Projected Evolutionary Lifting.

We remark that the Hamiltonian here is assumed to be Lipschitz continuous and concave, though not necessarily differentiable. While these hypotheses allow for the study of control problems where the Hamiltonian lacks regularity, the concavity assumption excludes the analysis of equations arising from, for instance, differential games.

The rest of the paper is organized as follows:
In Section \ref{sec:setup}, we introduce the notation and underlying stochastic control problem as well as the relevant structural assumptions. 
Section \ref{sec:struct_prop} introduces the Projected Evolutionary Lifting and discusses its properties.
In Section \ref{sec:pseudo_resolvent_extention}, we present the main result of the paper, Theorem \ref{thm:ext_all_lambda}, which is an existence and uniqueness result for solutions to the HJB equations for any $\lambda > 0$.
In Section \ref{sec:applications}, we apply our results to controlled stochastic wave equations and the stochastic heat equation with boundary control.
\section{Preliminaries and Setup}\label{sec:setup}
Throughout, let $H$ be a separable Hilbert space and $\Hbar$ a separable Banach space that is a control extension of $H$, i.e., we have $H \subseteq\Hbar$ such that $H$ is dense in $\Hbar$ and the inclusion $H \hookrightarrow \Hbar$ is continuous.
Furthermore, consider another separable Hilbert space $K$ acting as the control space.
Let $(\Omega, \cF, (\cF_t)_{t\geq0}, \bP)$ be a complete, filtered probability space satisfying the usual assumptions and let the separable Hilbert space $\Xi$ be the noise space.

For two separable Banach spaces $E_1$ and $E_2$, let $\cL(E_1, E_2)$ be the set of linear, bounded operators from $E_1$ to $E_2$. For an operator $A \in \cL(E_1, E_2)$, we denote its formal adjoint operator as $A^* \in \cL(E_1^*, E_2^*)$, where, if $E_1$ or $E_2$ are Hilbert spaces, we identify their dual space with themselves.

Let $B \in \mathcal{L}(K, E_1)$. Then, a function $f\colon E_1 \to E_2$ is \textit{$B$-directionally differentiable} at $x \in E_1$ in the direction $k \in K$ if the limit $\nabla^B f(x;\, k) = \lim_{s \to 0} \frac{f(x + sBk) - f(x)}{s}$ exists in $E_2$. 
The function is \textit{$B$-Gâteaux differentiable} if $k \mapsto \nabla^B f(x;k)$ is an element of $\mathcal{L}(K, E_2)$, and \textit{$B$-Fréchet differentiable} if the convergence is uniform for $\|k\|_K \leq 1$. 

For a separable Banach space $E$, we denote by $B_b(E)$ the space of bounded, Borel measurable functions $f\colon E \to \bR$ equipped with the supremum norm $\|f \|_{\infty} = \sup_{x \in E} |f(x)|$.
Moreover, we write $C(E)$ for the space of continuous real-valued functions on $E$ and denote the space of bounded continuous functions as $C_b(E)$. 

Let $C^{1}_b(E)$ be the set of all Fr\'echet differentiable functions with bounded, continuous derivatives.
Furthermore, we denote by $C^{1,B}_b(E)$ the set of functions in $C_b(E)$ that are $B$-Fréchet differentiable.
Analogously, we denote the space of uniformly continuous, bounded functions as $\UCb(E)$ and the set of uniformly continuous, bounded functions with uniformly continuous and bounded $B$-Fr\'echet differential as $\UCb^{1,B}(E)$.

\subsection{Setup of the Control Problem}
Throughout, we consider an operator $A$ generating a strongly continuous semigroup $\{ e^{tA} \}_{t\geq 0}$ on $H$, which can be extended to a strongly continuous semigroup $\{  \overline{e^{tA}} \}_{t\geq 0}$ on $\Hbar$ and is of type $\omega$, i.e., there exist constants $M \ge 1$ and $\omega \in \mathbb{R}$ such that $\|\overline{e^{tA}}\|_{\mathcal{L}(\overline{H})} \le M e^{\omega t}$ for all $t \ge 0$. Additionally, we consider an unbounded control operator $B\in \cL(K, \Hbar)$ and diffusion operator $G\in\cL(\Xi,H)$ such that the covariance operator $Q_t\coloneqq\int_0^t e^{sA}GG^*e^{sA^*}\dd s$ is trace-class for all $t>0$. 

We are then interested in control problems of the $\Hbar$-valued controlled Ornstein-Uhlenbeck (OU) process
\begin{equation}\label{eq:ctrl_spde}
\begin{cases}
    \dd X(s)= AX(s) \dd s+ Bu(s) \dd s +G \dd W(s), \quad s\in(0,\infty),\\
    X(0)=x\in H,
\end{cases}
\end{equation}
where $W$ is a $\Xi$-valued cylindrical Wiener process on the underlying filtered probability space $(\Omega, \cF, (\cF_t)_{t\geq0}, \bP)$ and, for a closed and bounded subset $U\subseteq K$, the control $u$ is taken from the space of admissible controls
\begin{equation}\label{eq:adm_ctrls}
    \cU\coloneqq\big\{ u\colon [0,\infty)\times \Omega \to U \;\big|\; u \text{ is progressively measurable} \big\}.
\end{equation}

Then, the mild solution to the controlled SPDE in equation \eqref{eq:ctrl_spde} is given by the implicit equation
\begin{equation}\label{eq:mild_sol_ctrl_spde}
    X(s) = e^{sA}x+\int_0^s\overline{e^{(s-r)A}}B u(r) \dd r +\int_0^s e^{(s-r)A}G \dd W(r).
\end{equation}
For a time $t\geq0$, initial state $x\in H$, and control $u\in \cU$, we denote the mild solution as $X(t;\, x,u)$.

As usual, the uncontrolled Ornstein-Uhlenbeck (OU) process $Z(\cdot \,;\, x)$ on $H$ is the mild solution to the equation
\begin{equation}\label{eq:ou_equation}
    \begin{cases}
        \dd Z(s) = AZ(s) \dd t + G \dd W(s), \quad s\in(0,\infty),\\
        Z(0)=x \in H,
    \end{cases}
\end{equation}
given explicitly by $Z(t;\, x)= e^{tA}x + W_{A}(t)$, where $W_{A}(t) \coloneqq \int_{0}^te^{(t-r)A}G \dd W(r)$,
with the associated transition semigroup acting on $B_{b}(H)$ given by
\begin{equation}\label{eq:ou_semigroup}
    P_{t}[\phi](x) \coloneqq \bE \big[ \phi\big(Z(t;\, x)\big) \big] = \int_{H} \phi\Big( e^{tA}x+y \Big) \,\cN(0,Q_{t})(\dd y)
\end{equation}
for $t\geq 0$ and $x\in H$.

Note that $P_t$ admits a canonical extension to $\Hbar$ that we, with a slight abuse of notation, also denote as $P_t$, cf.\ \cite[Corollary 3.2]{BoGo25}.

Having established the underlying dynamical system, we then consider the following cost functional to minimize:
First, let $\ell_0\colon \Hbar \to \bR$ be a cost associated with the state variable. To that end let $\ell_0$ be measurable and bounded. 
Furthermore, consider the cost of the control $\ell_1\colon U \to \bR$ being measurable and bounded.
Then, for some discount factor $\lambda>0$, the cost functional to be minimized is
\begin{equation}\label{eq:ctrl_cost_fctl}
J(x;\, u) \coloneqq \mathbb{E} \bigg(\int_0^\infty e^{-\lambda s}\Big[\ell_0\big(X(s; x, u)\big) + \ell_1\big(u(s)\big)\Big] \dd s \bigg).
\end{equation}
Note that, given the conditions on $\ell_0$ and $\ell_1$, the cost functional \eqref{eq:ctrl_cost_fctl} is well-defined and bounded from below.
Thus, the value function of the optimal control problem 
\begin{equation}\label{eq:ctrl_val_fct}
V(x)\coloneqq \inf_{u\in \mathcal{U}} J(x;\, u)
\end{equation}
is also well-defined for any $x \in \Hbar$.

\begin{remark}
    The considered control problem can be adapted to the case in which $\ell_{0}$ has polynomial growth in a straightforward manner. However, this adaptation would require some additional assumptions like the semigroup being of negative type and a change of topology to work on a suitable weighted space. 
\end{remark}

\begin{remark}
    Note that we define the running cost $\ell_{0}$ directly on the extended space $\overline{H}$ to ensure that $J(x;u)$ and $V(x)$ are well-defined even if the trajectories exit $H$. 
\end{remark}

\section{Projected Evolutionary Lifting and its Properties}\label{sec:struct_prop}

In this section, we introduce and discuss the properties of the Projected Evolutionary Lifting. Furthermore, we provide results that allow is to study our control problem more closely.

\subsection{Projected Evolutionary Lifting}\label{sec:pel}

We now turn to the Projected Evolutionary Lifting (PEL), which is a central concept of study in this work and previously had been introduced in \cite{GoMa25,BoGo25}. Intuitively, for each element $x \in \Hbar$, the PEL is a part of the trajectory along the semigroup generated by $A$, i.e., for some selection operator $\cP$, the lifting is a trajectory $t \mapsto \cP e^{tA} x$.

More formally, let $\cP \in \cL(H, H)$ and assume that, for any $t>0$, the operator $\cP e^{tA} \colon H\rightarrow H$ can be continuously extended to an operator $\overline{\cP e^{tA}} \colon \Hbar\rightarrow H$. Note, that under this assumption the map $t \mapsto \overline{\cP e^{tA}} x$ is continuous in $H$
and the identity $\overline{\cP e^{tA}} x= \overline{\cP e^{sA}} \cdot \overline{e^{(t-s)A}} x$ holds, cf.\ \cite[Lemma 2.8]{GoMa25}.

Given this regularity, we define the space of admissible paths.

\begin{definition}\label{def:path_space}
Let $A$ be the generator of a strongly continuous semigroup and $\cP \in \cL(H,H)$ a selection operator such $\cP e^{tA} \colon H\rightarrow H$ can be continuously extended to $\overline{\cP e^{tA}} \colon \Hbar\rightarrow H$. Then, we define the \emph{space of admissible paths} as
\begin{equation}
    \cC^{\cP}_A(0,\infty;\, H) \coloneqq \Big\{ f\in C (0,\infty;\, H) \;\Big|\; f(t) = \overline{\cP e^{tA}} x \text{ for all } t>0 \text{ and an } x\in \Hbar \Big\}.
\end{equation}
When the co-domain is clear from context, we simply denote the space as $\cC^{\cP}_A$.
\end{definition}

\begin{definition}[Projected Evolutionary Lifting]
    Let $A$ be the generator of a strongly continuous semigroup and $\cP \in \cL(H,H)$ a selection operator such $\cP e^{tA} \colon H\rightarrow H$ can be continuously extended to $\overline{\cP e^{tA}} \colon \Hbar\rightarrow H$.
    Then, the \emph{Projected Evolutionary Lifting} (PEL) of $x\in \overline{H}$ is the trajectory $t \mapsto y^P_{x}(t)=\overline{Pe^{tA}}x$, which is an element of the set of admissible paths $\cC^P_A((0,\infty);\, H)$.
    We denote the \emph{lifting map} as $\Upsilon_\infty^{\cP}(x) = y^{\cP}_x$.
\end{definition}

\begin{remark}
    As shown in \cite[Lemma 2.10]{GoMa25}, the lifting map $\Upsilon_\infty^{\cP}(x)$ is surjective and continuous when the path space $\cC^{\cP}_A$ is equipped with the topology of uniform convergence on compact sets.
\end{remark}

As usual in infinite time horizon problems, we are interested in exponentially discounted trajectories. More specifically, we are interested in trajectories from the following weighted Hilbert space.

\begin{definition}\label{def:L2_rho}
For some constant $\rho>0$, let
    \begin{equation}
        L^2_\rho(0,\infty;\, H) \coloneqq \bigg\{ f\colon(0,\infty)\to H \;\bigg|\; \|f\|_{L_\rho^2}^2 \coloneqq \int_0^{\infty} e^{-\rho t} \left| f(t)\right|^2_H \dd t < \infty \bigg\}.
    \end{equation}
If the co-domain is clear from context, we simply write $L^2_\rho$.
\end{definition}

To ensure that our lifted trajectories are in $L^2_\rho$, we need to impose some assumption on the growth of the trajectories. Thus, we require the semigroup to satisfy an exponential stability-type condition, i.e., we assume there exist $\omega\in\bR$ and $\eta \in [0,1/2)$ such that $\big\| \overline{Pe^{tA}} x \big\|_H\leq e^{\omega t} t^{-\eta} \|x\|_{\overline{H}}$ holds for all $x \in \overline{H}$.
In this case, \cite[Lemma 2.13]{GoMa25} shows that the lifting map $\Upsilon^P_{\infty}$ is continuous embedding from $\overline{H}$ into $L_\rho^2(0,\infty;\, H)$.
Note that, unlike in finite-horizon problems, where continuous trajectories are naturally bounded, stationary dynamics necessitate some damping weight $e^{-\rho t}$ to arrest infinite-horizon growth, thereby preserving the topological well-posedness of $\Upsilon^P_\infty$.

Using the tools above, we can now define the class of liftable functions.
\begin{definition}\label{def:SP}
We define the space of real-valued \emph{liftable function} on $\Hbar$ to be
\begin{equation}\label{def-funz-traj}
    \mathcal{S}^{\cP}_\infty(\overline{H}) \coloneqq \Big\{\phi \colon \Hbar \to \bR \;\Big|\;  \text{there exists } \hat{\phi}\in B_{b}(\mathcal{C}^{\cP}_A(0,\infty;\, H)) \text{ s.t. } \phi(x)=\hat{\phi}(y^{\cP}_{x}) \text{ for all } x \in \Hbar \Big\}.
\end{equation}
\end{definition}

\begin{remark}\label{rem:first_props}
Given Definition \ref{def:SP}, we immediately have the following properties:
\begin{enumerate}[(i)]
    \item The inclusion $\mathcal{S}^{\cP}_{\infty}(\overline{H}) \subseteq B_b(\overline{H})$ holds.

    \item\label{item:rem:first_props} Every $\phi\in \mathcal{S}^{\cP}_\infty(\overline{H})$ admits the decomposition $\phi=\hat \phi \circ \Upsilon^{\cP}_\infty$.
\end{enumerate}
\end{remark}

Later, the set of liftable functions that are bounded and uniformly continuous play a key role, for which we introduce the following notation.
\begin{definition}\label{def:liftable_buc}
    We denote the set of liftable, uniformly continuous, and bounded functions on $\Hbar$ as
    \begin{equation}
        \LUC(\Hbar) \coloneqq \UCb(\overline{H}) \cap \mathcal{S}^{\cP}_\infty(\overline{H})
    \end{equation}
    and, analogously,
    \begin{equation}
        \LUC^{1, B} (\Hbar) \coloneqq \UCb^{1,B}(\overline{H}) \cap \mathcal{S}^{\cP}_\infty(\overline{H})
    \end{equation}
    for the set of liftable, uniformly continuous, bounded, and $B$-Fr\'echet differentiable functions.
\end{definition}

\begin{remark}
    We want to point out that, in general, Definition \ref{def:liftable_buc} is a proper intersection as can be seen from the following example:
    Let $H = \Hbar = \bR^2$, $A=0$, and $\cP$ be the projection on the second component, i.e., $\cP(x) = \pi_2(x)$.
    Then, $\Upsilon^{\pi_2}_\infty (x)$ is the constant function $x_2$.
    Now, consider $\phi (x) = \sin (x_1)$, which is an element of $\UCb(\bR^2)$. However, we cannot find any $\hat\phi \in B_b(\cC^{\pi_2}_A(0, \infty;\, \bR^2))$ such that $\phi (x) = \hat\phi \circ \Upsilon^{\pi_2}_\infty (x)$ as this would imply that $\phi$ only depends on $x_2$.
\end{remark}

\subsection{First Properties}

First, we recall some important properties of the set of liftable functions $\mathcal{S}^{\cP}_\infty(\Hbar)$, which have been shown in \cite[Proposition 2.15]{GoMa25}.

\begin{proposition}\label{prop-adjoint}
Consider the PEL as constructed in Section \ref{sec:pel}. Then:
\begin{enumerate}[(i)]
    \item\label{item:prop-adjoint} If $\cP$ extends continuously to $\overline{{\cP}}:\overline{H}\to H$ with $\operatorname{Im} \overline{{\cP}}=\operatorname{Im} {\cP}$, then any function in $B_b^{\cP}(H)$ can be extended to a function in $\mathcal{S}^{\cP}_\infty(\Hbar)$ by setting $\phi(x) = \hat\phi (\overline\cP x)$ for any $x \in \Hbar$. If $\cP$ also commutes with $A$, this extension is an isomorphism, i.e., $B_b^{\cP}(H) \cong  \mathcal{S}^{\cP}_\infty(\Hbar)$.
    \item The adjoint operator $(\Upsilon^{\cP}_\infty)^* \colon L^2_\rho(0,\infty;H) \to H$ is explicitly given by $(\Upsilon^{\cP}_\infty)^*z(\cdot) = \int_{0}^{\infty}e^{-\rho s}e^{sA^*}{\cP}^*z(s) \dd s$.
\end{enumerate}
\end{proposition}

The next result deals with the completeness of the space of liftable functions $\mathcal{S}^{\cP}_\infty(\Hbar)$.
\begin{proposition}
The space $\mathcal{S}^{\cP}_\infty(\Hbar)$ is a Banach space under the supremum norm.
\end{proposition}

\begin{proof}
Let $(\phi_n)_{n \in \mathbb{N}}$ be a Cauchy sequence in $\mathcal{S}^{\cP}_\infty(\overline{H})$. By Remark \ref{rem:first_props} \ref{item:rem:first_props}, there exists a sequence $\hat{\phi}_n \in B_b(\mathcal{C}^P_A)$ such that $\phi_n = \hat{\phi}_n \circ \Upsilon^{\cP}_\infty$. As the lifting map $\Upsilon^{\cP}_\infty$ is surjective onto $\mathcal{C}^{\cP}_A$, we find
\begin{equation}
    \|\phi_n - \phi_m\|_\infty = \sup_{x \in \overline{H}} |\hat{\phi}_n(\Upsilon^{\cP}_\infty(x)) - \hat{\phi}_m(\Upsilon^{\cP}_\infty(x))| = \sup_{y \in \mathcal{C}^{\cP}_A} |\hat{\phi}_n(y) - \hat{\phi}_m(y)| = \|\hat{\phi}_n - \hat{\phi}_m\|_{B_b}.
\end{equation}
This, in turn, implies $(\hat{\phi}_n)_{n \in \mathbb{N}}$ is Cauchy in $B_b(\mathcal{C}^{\cP}_A)$, which is complete. Thus, we have that the limit $\hat{\phi} \in B_b(\mathcal{C}^{\cP}_A)$ exists and that $\hat{\phi}_n \to \hat{\phi}$.
Setting $\phi \coloneqq \hat{\phi} \circ \Upsilon^{\cP}_\infty$, which inherently belongs to $\mathcal{S}^{\cP}_\infty(\overline{H})$, we finally arrive at
\begin{equation}
    \|\phi_n - \phi\|_\infty = \sup_{y \in \mathcal{C}^{\cP}_A} |\hat{\phi}_n(y) - \hat{\phi}(y)| \to 0,
\end{equation}
as $n \to \infty$.
\end{proof}

\subsection{Partial Smoothing}

It is well known, see, e.g., \cite[Theorem 9.26]{DPZa14}, that, if the condition $\operatorname{Im}e^{tA}\subseteq \operatorname{Im} Q_{t}^{1/2}$ is satisfied, the transition semigroup $P_{t}$ maps bounded and measurable functions into differentiable ones. However, such a condition turns out to be very restrictive and, in particular, is not verifiable in our examples in Section \ref{sec:applications}. For this reason, a weaker condition was introduced in \cite{GoMa17}. More precisely, it is shown that, if the condition $\operatorname{Im}(\cP e^{tA}B) \subseteq \operatorname{Im} \left(\cP Q_{t} \cP^*\right)^{1/2}$ is satisfied, the transition semigroup regularizes functions that depend on the state only through the action of the projection $\cP$ in the directions given by an operator $B$. However, as pointed out in \cite{GoMa25, BoGo25} this condition is not enough to deal with control problems in which the running cost depends on the state variable. This limitation motivates the introduction of the following, more dynamic assumption.

\begin{hypothesis}\label{ipo-smoothing-lifting}
We additionally assume that
\begin{enumerate}[(i)]
    \item for any $t> 0$ and $k \in K$, we have $\Upsilon^{\cP}_\infty\overline{e^{tA}}Bk \in \operatorname{Im} \left(\Upsilon^{\cP}_\infty Q_{t} (\Upsilon^{\cP}_\infty)^*\right)^{1/2}$ such that the lifted operator $\widehat\Lambda^{\cP,B}(t)k \coloneqq \big(\Upsilon^{\cP}_\infty Q_{t} (\Upsilon^{\cP}_\infty)^*\big)^{-1/2} \Upsilon^{\cP}_\infty e^{tA}Bk$ is well-defined via the Closed Graph Theorem,
    \item there exist constants $\kappa_0>0$ and $\gamma \in (0,1)$ such that the singularity at $0$ is integrable, i.e., we assume that  $\|\widehat\Lambda^{\cP,B}(t)\|_{\mathcal{L}(K,L^2_\rho)} \le \kappa_0 (t^{-\gamma}\vee 1)$ for all $t > 0$.
\end{enumerate}
\end{hypothesis}

Provided this structural condition is met, we can show the following, Bismut--Elworthy--Li-type formula for lifted partial smoothing, which explicitly provides directional derivatives and the proof of which can be found in \cite[Proposition 2.6]{GoMa25}

\begin{proposition}\label{prop-lift-partial-smoothing}
Under Hypothesis \ref{ipo-smoothing-lifting}(i), the semigroup $P_t$ maps any function $\phi\in \mathcal{S}^P_\infty(\overline{H})$ to a $B$-Fréchet differentiable function. Furthermore, for $t>0$ and $x\in \overline{H}$, we have
\begin{gather}\label{eq:formulader-gen-Pnew}
\nabla^B P_{t}[\phi](x)k = \int_{L^2_\rho}\hat\phi\big(z_1+ \Upsilon^{\cP}_\infty x \big) \langle\widehat\Lambda^{\cP,B}(t) k, (\Upsilon^{\cP}_\infty Q_t(\Upsilon^{\cP}_\infty)^*)^{-1/2} z_1\rangle_{L^2_\rho} \,\mathcal{N}(0,\Upsilon^{\cP}_\infty Q_t(\Upsilon^{\cP}_\infty)^*)(\!\dd z_1)\\
= \mathbb{E} \Big[\hat\phi \big(\Upsilon^{\cP}_\infty X(t; x)\big) \langle\widehat\Lambda^{\cP,B}(t) k, (\Upsilon^{\cP}_\infty Q_t(\Upsilon^{\cP}_\infty)^*)^{-1/2} \Upsilon^{\cP}_\infty W_A(t)\rangle_{L^2_\rho} \Big].
\end{gather}
Moreover, we have
\begin{equation}
    \Big| \langle\nabla^B P_t[\phi](x), k\rangle\Big| \leq \|\widehat\Lambda^{P,B}(t)\|_{\mathcal{L}(K,L^2_\rho)} \Vert\phi\Vert_\infty |k|.
\end{equation}
\end{proposition}

\subsection{Resolvent Operator for Liftable Functions}
In preparation for our main results in Section \ref{sec:pseudo_resolvent_extention}, we proceed to show some properties of the resolvent of the uncontrolled OU semigroup acting on liftable functions.
Throughout, for $\lambda >0$, we will use the following notation for the resolvent 
\begin{equation}\label{def:resolvent}
    T_\lambda \psi(x) \coloneqq \int_0^\infty e^{-\lambda t} P_t[\psi](x) \dd t.
\end{equation}

The first result of this section shows that the resolvent $T_\lambda$ is actually well defined on $S^{\cP}_\infty$.
\begin{lemma}\label{lemma:Tlambda_in_Sp}
For every $\lambda > 0$, $\psi \in \mathcal{S}^{\cP}_\infty(\overline{H})$, and $x \in \Hbar$, the resolvent operator
\begin{equation}
T_\lambda \psi(x) = \int_0^\infty e^{-\lambda t} P_t[\psi](x) \dd t
\end{equation}
maps $\mathcal{S}^P_\infty(\overline{H})$ into itself, i.e., $T_\lambda \psi \in \mathcal{S}^P_\infty(\overline{H})$.
\end{lemma}

\begin{proof}
Since $\psi \in \mathcal{S}^{\cP}_\infty(\overline{H})$, by definition, there exists a bounded and Borel-measurable function $\hat{\psi}\colon L^2_\rho \to \mathbb{R}$ such that $\psi(x) = \hat{\psi}(\Upsilon^{\cP}_\infty (x))$ for every $x \in \overline{H}$. Our goal is to explicitly construct a lifted function $\widehat{T_\lambda \psi} \colon L^2_\rho \to \mathbb{R}$ that is bounded, Borel-measurable, and satisfies $T_\lambda \psi(x) = \widehat{T_\lambda \psi}(\Upsilon^{\cP}_\infty (x))$.

First, let $S_t \colon L^2_\rho \to L^2_\rho$ be the left-shift operator, meaning, for each $t>0$, we have
\begin{equation}
(S_t f)(s) \coloneqq f(t+s), \quad \text{for } s \in (0, \infty).
\end{equation}
We now show that $S_t$ is a bounded linear operator on $L^2_\rho$. Computing its norm, we obtain
\begin{align}
\|S_t f\|_{L^2_\rho}^2 &= \int_0^\infty e^{-\rho s} |f(t+s)|_H^2 \dd s
= \int_t^\infty e^{-\rho (r-t)} |f(r)|_H^2 \dd r \\
&= e^{\rho t} \int_t^\infty e^{-\rho r} |f(r)|_H^2 \dd r
\le e^{\rho t} \|f\|_{L^2_\rho}^2.
\end{align}
Thus $\|S_t\|_{\mathcal{L}(L^2_\rho;\, L^2_\rho)} \le e^{\rho t}$, which ensures that the shift is a well-defined and continuous, hence Borel-measurable, operator for every $t > 0$.

Next, we evaluate the action of the lifting map $\Upsilon^{\cP}_\infty$ on the evolved state $\overline{e^{tA}}x$. By definition of the lifting and the strong semigroup property, for every $s > 0$, we have
\begin{equation}
    \big[ \Upsilon^{\cP}_\infty(\overline{e^{tA}}x) \big] (s) = \overline{\cP e^{sA}} (\overline{e^{tA}}x) = \overline{\cP e^{(s+t)A}}x = \big[\Upsilon^{\cP}_\infty x\big](t+s) = \big[S_t (\Upsilon^{\cP}_\infty x)\big](s)
\end{equation}
in the $L^2_\rho$ sense.

Recall that the action of the OU semigroup $P_t$ on $\psi$ is given by $P_t[\psi](x) = \mathbb{E}[\psi(\overline{e^{tA}}x + W_A(t))]$. 
Thus, the lifted random variable $\Upsilon^{\cP}_\infty W_A(t)$ induces the Gaussian measure $\mathcal{N}(0, \Sigma_t)$ on $L^2_\rho$, where the covariance operator is given by $\Sigma_t = \Upsilon^{\cP}_\infty Q_t(\Upsilon^{\cP}_\infty)^*$. 
Rewriting in terms of the lifted function $\hat{\psi}$ and the shift operator $S_t$, we find
\begin{align*}
    P_t[\psi](x) &= \mathbb{E}[\hat{\psi}(\Upsilon^P_\infty(\overline{e^{tA}}x) + \Upsilon^P_\infty W_A(t))]= \int_{L^2_\rho} \hat{\psi}(S_t (\Upsilon^P_\infty x) + z) \, \mathcal{N}(0,\Sigma_t)(\!\dd z).
\end{align*}

Motivated by this representation, for every trajectory $\eta \in L^2_\rho(0, \infty;\, H)$, we define the lifted function $\widehat{T_\lambda \psi}\colon L^2_\rho(0,\infty;\, H) \to \mathbb{R}$ as
\begin{equation}
    \widehat{T_\lambda \psi}(\eta) \coloneqq \int_0^\infty e^{-\lambda t} \bigg( \int_{L^2_\rho} \hat{\psi}(S_t \eta + z) \, \mathcal{N}(0,\Sigma_t)(dz) \bigg) \dd t.
\end{equation}
Evaluating at $\eta = \Upsilon^P_\infty (x)$, we exactly obtain $T_\lambda \psi(x) = \widehat{T_\lambda \psi}(\Upsilon^P_\infty (x))$. 
Since $\hat{\psi}$ is bounded and $\mathcal{N}(0,\Sigma_t)$ is a probability measure for every $t$, the inner integral is bounded by $\|\hat{\psi}\|_\infty$. Integrating over $t$, we obtain
\begin{equation}
    \big|\widehat{T_\lambda \psi}(\eta) \big| \le \|\hat{\psi}\|_\infty \int_0^\infty e^{-\lambda t} \dd t = \frac{1}{\lambda}\|\hat{\psi}\|_\infty < \infty.
\end{equation}
Note that the map $\eta \mapsto S_t \eta + z$ is continuous in $L^2_\rho$. Since $\hat{\psi}$ is Borel-measurable, the composition $\eta \mapsto \hat{\psi}(S_t \eta + z)$ is measurable. By the Fubini-Tonelli Theorem, integrating with respect to the measure $\mathcal{N}(0,\Sigma_t)$ and the Lebesgue measure preserves Borel-measurability. Since $\widehat{T_\lambda \psi}$ is well-defined, bounded, and measurable, we conclude that indeed $T_\lambda \psi \in \mathcal{S}^{\cP}_\infty(\overline{H})$.
\end{proof}

Next, we show that $T_\lambda$ maps $\UCb(\overline{H})$ functions to $\UCb(\overline{H})$ functions.
\begin{lemma}\label{lemma:Tlambda_in_UCb}
For every $\lambda > 0$ and every $\psi \in \UCb(\overline{H})$, the resolvent operator $T_\lambda \psi$ belongs to $\UCb(\overline{H})$.
\end{lemma}

\begin{proof}
Let $\psi \in \UCb(\overline{H})$. Since $\psi$ is uniformly continuous, it admits a modulus of continuity $\rho_\psi \colon [0, \infty) \to [0, \infty)$ defined by
\begin{equation}
    \rho_\psi(\delta) \coloneqq \sup_{|x-y|_{\overline{H}} \le \delta} |\psi(x) - \psi(y)|.
\end{equation}
By definition, $\rho_\psi$ is a non-decreasing function and $\lim_{\delta \downarrow 0} \rho_\psi(\delta) = 0$. Since $\psi$ is bounded, we have $\rho_\psi(\delta) \le 2\|\psi\|_\infty$ for all $\delta \ge 0$.

Boundedness of $T_\lambda \psi$ follows from $|P_t[\psi](x)| \le \|\psi\|_\infty$ for all $t \ge 0$ and $x \in \overline{H}$, as
\begin{equation}
    \|T_\lambda \psi\|_\infty \le \int_0^\infty e^{-\lambda t} \| \psi \|_\infty \dd t = \frac{1}{\lambda} \|\psi\|_\infty < \infty.
\end{equation}

To prove uniform continuity, we explicit construct the modulus of continuity for $T_\lambda \psi$: For any $x, y \in \overline{H}$ such that $|x - y|_{\overline{H}} \le \delta$, consider
\begin{equation}
    |T_\lambda \psi(x) - T_\lambda \psi(y)| \le \int_0^\infty e^{-\lambda t}  \mathbb{E}\Big[ \Big| \psi\big(\overline{e^{tA}}x + W_A(t)\big) - \psi\big(\overline{e^{tA}}y + W_A(t)\big)\Big| \Big]  \dd t.
\end{equation}
Using the modulus of continuity $\rho_\psi$ and that it is non-decreasing as well as the bound on the extended semigroup $\overline{e^{tA}}$, we estimate
\begin{align}
    \Big|\psi(\overline{e^{tA}}x + W_A(t)) - \psi(\overline{e^{tA}}y + W_A(t))\Big| &\le \rho_\psi\big(\big|\overline{e^{tA}}x - \overline{e^{tA}}y\big|_{\overline{H}}\big)\\
    &\le \rho_\psi\big( \|\overline{e^{tA}}\|_{\mathcal{L}(\overline{H})} |x - y|_{\overline{H}} \big)\\
    &\le \rho_\psi( M e^{\omega t} \delta ).
\end{align}
Taking the expectation and substituting back into the integral, we find that the difference is bounded independently of the specific choice of $x$ and $y$ as
\begin{equation}
|T_\lambda \psi(x) - T_\lambda \psi(y)| \le \int_0^\infty e^{-\lambda t} \rho_\psi( M e^{\omega t} \delta ) \dd t.
\end{equation}
Taking the supremum over all $x,y \in \overline{H}$ with $|x - y|_{\overline{H}} \le \delta$, we define the candidate modulus of continuity $\tilde{\rho}(\delta)$ for $T_\lambda \psi$ as
\begin{equation}\label{eq:ucb_modulus}
\tilde{\rho}(\delta) \coloneqq \sup_{|x-y|_{\overline{H}} \le \delta} |T_\lambda \psi(x) - T_\lambda \psi(y)| \le \int_0^\infty e^{-\lambda t} \rho_\psi( M e^{\omega t} \delta ) \dd t.
\end{equation}

To establish uniform continuity, we must show that $\lim_{\delta \downarrow 0} \tilde{\rho}(\delta) = 0$. To that end, we apply Lebesgue's Dominated Convergence Theorem (DCT) to the integral sequence parameterized by $\delta$. For any fixed $t > 0$, as $\delta \downarrow 0$, the argument $M e^{\omega t} \delta \to 0$. By the uniform continuity of $\psi$, $\rho_\psi( M e^{\omega t} \delta ) \to 0$. Furthermore, for all $\delta \ge 0$ and $t > 0$, the integrand is strictly bounded by
\begin{equation}
    e^{-\lambda t} \rho_\psi( M e^{\omega t} \delta ) \le e^{-\lambda t} 2\|\psi\|_\infty.
\end{equation}
By the DCT, we can exchange the limit with the integral, yielding that $\lim_{\delta \downarrow 0} \tilde{\rho}(\delta) = \int_0^\infty 0 \dd t = 0$, showing that $T_\lambda \psi \in \UCb(\overline{H})$.
\end{proof}

\begin{definition}
    We say that the \emph{partial smoothing hypothesis} holds if there exist constants $\kappa_0 > 0$ and $\gamma \in (0, 1)$ such that, for every $\phi \in \mathcal{S}^{\cP}_\infty(\overline{H})$ and $t>0$, the function $P_t[\phi]$ is $B$-differentiable, and, for all $t > 0$ and $x \in \overline{H}$, we have
    \begin{equation}\label{eq:smoothing_bound_recalled}
        \|\nabla^B P_t[\phi](x)\|_{\mathcal{L}(K, \mathbb{R})} \le \kappa_0 (1 \vee t^{-\gamma}) \|\phi\|_\infty.
    \end{equation}
\end{definition}
Note that due to Proposition \ref{prop-lift-partial-smoothing}, Hypothesis \ref{ipo-smoothing-lifting} is a sufficient condition for the partial smoothing hypothesis to hold.

\begin{lemma}\label{lemma:Tlambda_B_diff}
Let the partial smoothing hypothesis hold.
Then, for every $\lambda > 0$ and $\psi \in \mathcal{S}^{\cP}_\infty(\overline{H})$, the resolvent operator $T_\lambda \psi$ is $B$-G\^ateaux differentiable on $\overline{H}$, and its derivative is given by 
\begin{equation}\label{eq:Tlambda_derivative_formula}
    \nabla^B T_\lambda \psi(x) = \int_0^\infty e^{-\lambda t} \nabla^B P_t[\psi](x) \dd t.
\end{equation}
Furthermore, the derivative is bounded, i.e., we have $\|\nabla^B T_\lambda \psi\|_\infty \le C_{\lambda, \gamma} \|\psi\|_\infty$ with the constant $C_{\lambda, \gamma} \coloneqq \int_0^\infty e^{-\lambda t} \kappa_0 (1 \vee t^{-\gamma}) \dd t < \infty$.
\end{lemma}

\begin{proof}
Let $\psi \in \mathcal{S}^{\cP}_\infty(\overline{H})$, fix a state $x \in \overline{H}$ and arbitrary direction $k \in K$. 
Consider the difference quotient of $T_\lambda \psi$ at $x$ in the direction $Bk$ with step size $h \in \mathbb{R} \setminus \{0\}$,
\begin{equation}
\frac{T_\lambda \psi(x + h B k) - T_\lambda \psi(x)}{h} = \int_0^\infty e^{-\lambda t} \frac{P_t[\psi](x + h B k) - P_t[\psi](x)}{h} \, dt.
\end{equation}
To evaluate the limit as $h \to 0$, we apply Lebesgue's Dominated Convergence Theorem (DCT). 
By the partial smoothing hypothesis, the map $x \mapsto P_t[\psi](x)$ is $B$-differentiable for every fixed $t > 0$. Therefore, as $h \to 0$, the integrand converges pointwise for every $t > 0$:
\begin{equation}
    \lim_{h \to 0} e^{-\lambda t} \frac{P_t[\psi](x + h B k) - P_t[\psi](x)}{h} = e^{-\lambda t} \langle \nabla^B P_t[\psi](x), k \rangle_K.
\end{equation}
Note that, for any fixed $t > 0$, the map $s \mapsto P_t[\psi](x + s B k)$ is differentiable. By the Mean Value Theorem, for each $h \neq 0$, there exists some $\theta \in (0, 1)$ depending on $x, h, k$, and $t$ such that:
\begin{equation}
    \frac{P_t[\psi](x + h B k) - P_t[\psi](x)}{h} = \langle \nabla^B P_t[\psi](x + \theta h B k), k \rangle_K.
\end{equation}
Taking the absolute value and applying the partial smoothing bound \eqref{eq:smoothing_bound_recalled}, we obtain an estimate that is uniform w.r.t.\ $h$ and $\theta$. We have
\begin{align}
    \Bigg| \frac{P_t[\psi](x + h B k) - P_t[\psi](x)}{h} \Bigg| &\le \|\nabla^B P_t[\psi](x + \theta h B k)\|_{\mathcal{L}(K, \mathbb{R})} \|k\|_K \\
    &\le \kappa_0 (1 \vee t^{-\gamma}) \|\psi\|_\infty \|k\|_K.
\end{align}
Multiplying by the discount factor $e^{-\lambda t}$, we find that the integrand is strictly dominated by the function
\begin{equation}
    g(t) \coloneqq e^{-\lambda t} \kappa_0 (1 \vee t^{-\gamma}) \|\psi\|_\infty \|k\|_K.
\end{equation}
Since $\gamma \in (0, 1)$, the singularity at $t=0$ is integrable and, since $\lambda > 0$, the exponential decay guarantees integrability at infinity. Consequently, $g \in L^1(0, \infty;\, \bR)$ and independent of $h$.

Since the conditions of the DCT are met, we can interchange the limit and the integral and find
\begin{equation}
    \lim_{h \to 0} \frac{T_\lambda \psi(x + h B k) - T_\lambda \psi(x)}{h} = \int_0^\infty e^{-\lambda t} \langle \nabla^B P_t[\psi](x), k \rangle_K \, dt.
\end{equation}
Because the right-hand side is a bounded linear functional acting on $k \in K$ due to the dominating function $g(t)$, this proves that $T_\lambda \psi$ is $B$-G\^ateaux differentiable on $\overline{H}$ with the given derivative formula.

Finally, to establish the boundedness, we consider the operator norm of the derivative. Since the Bochner integral of a bounded function preserves the norm inequality, we have
\begin{align}
    \|\nabla^B T_\lambda \psi(x)\|_{\mathcal{L}(K, \mathbb{R})} &= \sup_{\|k\|_K=1} \left| \int_0^\infty e^{-\lambda t} \langle \nabla^B P_t[\psi](x), k \rangle_K \, dt \right| \\
    &\le \int_0^\infty e^{-\lambda t} \|\nabla^B P_t[\psi](x)\|_{\mathcal{L}(K, \mathbb{R})} \dd t\\
    &\le \|\psi\|_\infty \int_0^\infty e^{-\lambda t} \kappa_0 (1 \vee t^{-\gamma}) \dd t.
\end{align}
Since this bound holds for all $x \in \overline{H}$, we obtain $\|\nabla^B T_\lambda \psi\|_\infty \le C_{\lambda, \gamma} \|\psi\|_\infty$.
\end{proof}

\begin{proposition}\label{prop:Tlambda_B_diff_UCb}
Assume the partial smoothing hypothesis holds with the operator $\widehat{\Lambda}^{\cP,B}(t)$ satisfying $\|\widehat{\Lambda}^{P,B}(t)\|_{\mathcal{L}(K, L^2_\rho)} \le \kappa_0 (1 \vee t^{-\gamma})$ for some $\gamma \in (0, 1)$. 
Then, for every $\lambda > 0$ and $\psi \in \LUC$, the derivative $\nabla^B T_\lambda \psi$ belongs to $\UCb(\overline{H};\, K)$.
\end{proposition}

\begin{proof}
Let $\psi \in \LUC$. From Lemma \ref{lemma:Tlambda_B_diff}, we know that $\nabla^B T_\lambda \psi$ exists, is bounded and given by
\begin{equation}
\nabla^B T_\lambda \psi(x) = \int_0^\infty e^{-\lambda t} \nabla^B P_t[\psi](x) \dd t.
\end{equation}
Since $\psi \in \UCb(\overline{H})$, it admits a modulus of continuity $\rho_\psi(\delta) \coloneqq \sup_{|x-y|_{\overline{H}} \le \delta} |\psi(x) - \psi(y)|$ with $\rho_\psi(\delta) \le 2\|\psi\|_\infty$ and $\lim_{\delta \downarrow 0} \rho_\psi(\delta) = 0$.

To prove uniform continuity, let $x, y \in \overline{H}$ such that $|x - y|_{\overline{H}} \le \delta$. Evaluating the norm of the difference of the derivatives in $K$, we find
\begin{equation}\label{eq:diff_derivative_integral}
    \|\nabla^B T_\lambda \psi(x) - \nabla^B T_\lambda \psi(y)\|_K \le \int_0^\infty e^{-\lambda t} \|\nabla^B P_t[\psi](x) - \nabla^B P_t[\psi](y)\|_K \dd t.
\end{equation}

By Proposition \ref{prop-lift-partial-smoothing}, the action of the $B$-derivative in any direction $k \in K$ is given by
\begin{equation}
    \langle \nabla^B P_t[\psi](x), k \rangle_K = \mathbb{E} \Big[ \psi(\overline{e^{tA}}x + W_A(t)) \langle \widehat{\Lambda}^{P,B}(t) k, \Sigma_t^{-1/2} \Upsilon^P_\infty W_A(t) \rangle_{L^2_\rho} \Big],
\end{equation}
where $\Sigma_t = \Upsilon^P_\infty Q_t (\Upsilon^P_\infty)^*$ is the covariance operator of the lifted noise. 
Taking the difference between the derivative at $x$ and $y$ evaluated in the direction $k$, we obtain
\begin{align*}
&|\langle \nabla^B P_t[\psi](x) - \nabla^B P_t[\psi](y), k \rangle_K| \\
&\le \mathbb{E} \Big[ \Big|\psi(\overline{e^{tA}}x + W_A(t)) - \psi(\overline{e^{tA}}y + W_A(t))\Big| \, \Big| \langle \widehat{\Lambda}^{P,B}(t) k, \Sigma_t^{-1/2} \Upsilon^P_\infty W_A(t) \rangle_{L^2_\rho} \Big| \Big].
\end{align*}
We can bound the first difference using the modulus of continuity exactly as before
\begin{equation}
    |\psi(\overline{e^{tA}}x + W_A(t)) - \psi(\overline{e^{tA}}y + W_A(t))| \le \rho_\psi( \|\overline{e^{tA}}\|_{\mathcal{L}(\overline{H})} |x - y|_{\overline{H}} ) \le \rho_\psi( M e^{\omega t} \delta ).
\end{equation}
Since this bound is deterministic, we can pull it out of the expectation and find
\begin{equation}
    |\langle \nabla^B P_t[\psi](x) - \nabla^B P_t[\psi](y), k \rangle_K| \le \rho_\psi( M e^{\omega t} \delta ) \, \mathbb{E} \left[ \left| \langle \widehat{\Lambda}^{P,B}(t) k, \Sigma_t^{-1/2} \Upsilon^P_\infty W_A(t) \rangle_{L^2_\rho} \right| \right].
\end{equation}
Recall that the random variable $\tilde{Z} = \Sigma_t^{-1/2} \Upsilon^P_\infty W_A(t)$ is a standard cylindrical Gaussian on $L^2_\rho$. Thus, the inner product $\langle h, \tilde{Z} \rangle_{L^2_\rho}$ is a real-valued, centered Gaussian with variance $\|h\|_{L^2_\rho}^2$. 
By the Cauchy-Schwarz inequality in $L^2(\Omega, \mathbb{P})$, the expectation of its absolute value is bounded by its standard deviation:
\begin{equation}
\mathbb{E} \left[ \left| \langle \widehat{\Lambda}^{P,B}(t) k, Z \rangle_{L^2_\rho} \right| \right] \le \|\widehat{\Lambda}^{P,B}(t) k\|_{L^2_\rho} \le \|\widehat{\Lambda}^{P,B}(t)\|_{\mathcal{L}(K, L^2_\rho)} \|k\|_K \le \kappa_0(1 \vee t^{-\gamma}) \|k\|_K.
\end{equation}
Taking the supremum over all $k \in K$ with $\|k\|_K = 1$, we obtain the uniform bound for the norm of the difference as
\begin{equation}\label{eq:diff_derivative_bound}
    \|\nabla^B P_t[\psi](x) - \nabla^B P_t[\psi](y)\|_K \le \kappa_0(1 \vee t^{-\gamma}) \rho_\psi( M e^{\omega t} \delta ).
\end{equation}

Substituting \eqref{eq:diff_derivative_bound} back into the integral  in equation \eqref{eq:diff_derivative_integral}, we find that the modulus of continuity of the derivative is bounded as
\begin{equation}
    \tilde{\rho}_{\nabla}(\delta) \coloneqq \sup_{|x-y|_{\overline{H}} \le \delta} \|\nabla^B T_\lambda \psi(x) - \nabla^B T_\lambda \psi(y)\|_K \le \int_0^\infty e^{-\lambda t} \kappa_0(1 \vee t^{-\gamma}) \rho_\psi( M e^{\omega t} \delta ) \dd t.
\end{equation}

Analogously to the proof of Lemma \ref{lemma:Tlambda_B_diff}, using Lebesgue’s Dominated Convergence Theorem, we now find that $\nabla^B T_\lambda \psi$ is uniformly continuous, which concludes the proof.
\end{proof}

\section{Main Results}\label{sec:pseudo_resolvent_extention}

In this section, we want to apply classical results about pseudo resolvents to show that the stationary Hamilton--Jacobi--Bellman (HJB) equation associated with the control problem introduced in Section \ref{sec:setup} admits a unique mild solution for all $\lambda >0$. 
More precisely, for $x \in \Hbar$, we consider the equation
\begin{equation}\label{eq:elliptic_HJB}
    \lambda v(x) = \cA v(x) +\ell_0(x)+ \Hmin \big(\nabla^B v(x)\big),
\end{equation}
where $\Hmin(p) \coloneqq \inf_{u \in \cU} \Hcv(p;\, u) = \inf_{u \in \cU} \big\{ \ip{p}{u} + \ell_1(u) \big\}$ and $\cA$ is the generator of the uncontrolled OU semigroup $P_t$.

As already the case for the underlying process, we are typically prevented from finding classical solutions. Thus, we consider solutions in the mild sense.
\begin{definition}\label{def:mild_sol_HJB}
    We say that a function $v\in C^{1,B}_{b}(\overline{H})$ is a \emph{mild solution} of equation \eqref{eq:elliptic_HJB} if, for every $x \in \Hbar$, it satisfies
    \begin{equation}\label{eq:hjb_mild_sol}
    v(x) = T^{\ell_0}_\lambda v (x) \coloneqq T_\lambda\big[\ell_0 + \Hmin(\nabla^B v)\big] (x) = \int_{0}^{\infty} e^{-\lambda t} P_{t}\Big[\ell_{0}(\cdot) + \Hmin \big(\nabla^{B}v(\cdot)\big)\Big](x) \dd t.
    \end{equation}
\end{definition}

As was shown in \cite[Theorem 4.6 and Lemma 4.8]{BoGo25}, there exists a $\lambda_0 > 0$ such that, for all $\lambda \geq \lambda_0$, the non-linear operator $T^{\ell_0}_\lambda v$ is the unique mild solution of equation \eqref{eq:elliptic_HJB} in the class of $\LUC^{1, B}(\Hbar)$.

Motivated by this result we introduce the solution mapping $R(\lambda)$.
\begin{definition}
    Let $\lambda_0 >0$ be as in \cite[Theorem 4.6]{BoGo25} and, for $\lambda \geq \lambda_0$ and $\psi \in \LUC(\Hbar)$, let $v_{\lambda,\psi}$ denote the unique mild solution to equation \eqref{eq:elliptic_HJB}, cf.\ \cite[Lemma 4.8]{BoGo25}. Then, the \emph{solution map} $R(\lambda)\colon \LUC(\Hbar) \to \LUC^{1, B}(\Hbar)$ is given by
    \begin{equation}
        \psi \mapsto R(\lambda) (\psi) \coloneqq v_{\lambda,\psi}.
    \end{equation}
\end{definition}

Our main result, Theorem \ref{thm:ext_all_lambda}, then intuitively states that, given that the minimum value Hamiltonian $\Hmin$ is concave, we can extend the the solution map $R(\lambda)$ to all $\lambda > 0$ on the set $\LUC(\Hbar)$.

\begin{theorem}\label{thm:ext_all_lambda}
    Let $\Hmin$ be concave. Then, for all $\lambda > 0$ and $\ell_0 \in \LUC(\Hbar)$, there exists a unique mild solution $v \in \LUC^{1,B}(\Hbar)$ to equation \eqref{eq:elliptic_HJB} such that
    \begin{equation}
        v(x) = R(\lambda)(\ell_0) = (\lambda - \cB)^{-1}(\ell_0),
    \end{equation}
    where $\cB \colon D(\cB) \subseteq \LUC^{1,B}(\Hbar) \to \LUC(\Hbar)$ is a unique, linear, maximally dissipative operator.
\end{theorem}

The proof of the above theorem can be found in Section \ref{sec:proof_main_thm}.
Generally, the proof combines two classical statements adapted to our lifted setting: Firstly, we show that the solution map $R(\lambda)$ is a pseudo resolvent and consequently, using some injectivity, we can find an unique operator $\cB$ such that the solution map $R(\lambda)$ is the resolvent of $\cB$, cf.\ \cite[Chapter VII.4, Theorem 1]{Yo65}.
In a second step, we show the dissipativity of $R(\lambda)$ and then use it to extend the solution map for all $\lambda>0$, cf.\ \cite[Chapter 1 Theorem 4.5]{Pa83}. For the proof of the second step, and in particular the dissipativity of $R(\lambda)$, we use that the minimum value Hamiltonian $\Hmin$ is concave.

\subsection{Auxiliary Results}

Before we proceed to the proof of Theorem \ref{thm:ext_all_lambda}, we provide some auxiliary results that allow us to characterize $R(\lambda)$ as an injective and dissipative pseudo resolvent.

\begin{proposition}\label{lem:resolvent_properties}
    Let $\lambda_0 > 0$ be as in \cite[Theorem 4.6]{BoGo25}. Then we have
    \begin{enumerate}[(i)]
        \item\label{item:id_of_resolvents} For all $\mu, \nu \geq \lambda_0$, the \emph{resolvent identity} holds, i.e., we have
        \begin{equation}\label{eq:prop:id_of_resolvents}
            R(\mu) = R(\nu) \circ \big( I + (\nu - \mu) R(\mu) \big),
        \end{equation}

        \item\label{item:injective_resolvent} For all $\mu \geq \lambda_0$ and $\phi \in \LUC(\Hbar)$, the map $\phi \mapsto R(\mu)\phi$ is injective,

        \item\label{item:lipschitz_resolvent} For all $\mu \geq \lambda_0$ and $\phi, \psi \in \LUC(\Hbar)$, we have
        \begin{equation}\label{eq:prop:contractive_sol_map}
            \| R(\mu)\phi - R(\mu)\psi \| \leq \frac1\mu \| \phi - \psi \|.
        \end{equation}
    \end{enumerate}
\end{proposition}

\begin{proof}
    We proof the three points separately.

    \noindent{\bfseries Proof of \ref{item:id_of_resolvents}:}
    We first show that the identity of resolvents holds for the linear resolvent $T_\lambda$.
    To that end, let $\mu, \nu >0$ and $\psi \in \LUC(\Hbar)$. Note that by definition of $P_t$ in equation \eqref{eq:ou_semigroup}, for any $\lambda >0$ and $x \in \overline{H}$, we have
    \begin{align}
        P_t[T_\lambda \psi](x) &= \bE\big[ T_\lambda \psi \big(\overline{e^{tA}}x+\overline{W(t)}\big) \big]\\
        &= \bE \bigg[ \int_0^\infty e^{-\lambda s} P_s [\psi] \big(\overline{e^{tA}}x+\overline{W(t)}\big) \dd s \bigg]\\
        &= \int_0^\infty e^{-\lambda s} P_{s+t} [\psi] (x) \dd s,
    \end{align}
    where the last equality follows from the Fubini-Tonelli Theorem and the semigroup property of $P_t$.
    Moreover, for any $\mu, \nu > 0$ with $\mu \neq \nu$, $\psi \in \LUC(\Hbar)$, and $x \in \overline{H}$, we find
    \begin{align}
        T_\mu(T_\nu \psi) (x) &= \int_0^\infty e^{-\mu t} \int_0^\infty e^{-\nu s} P_{t+s}[\psi] (x) \dd s \dd t\\
        &= \int_0^\infty e^{-\mu t} \int_t^\infty e^{-\nu (r-t)} P_{r}[\psi] (x) \dd r \dd t\\
        &= \int_0^\infty e^{-(\mu-\nu) t} \int_t^\infty e^{-\nu r} P_{r}[\psi] (x) \dd r \dd t\\
        &=\frac{1}{\nu - \mu}\bigg(-\int_0^\infty e^{-\nu r} P_r[\psi](x)\dd r + \int_0^\infty e^{-\mu r} P_r[\psi](x)\dd r \bigg)\\
        &= \frac{1}{\nu - \mu} \big(- T_\nu \psi (x) + T_\mu \psi (x) \big),
    \end{align}
    where we used $r=t+s$ for the second equality and integration by parts for the fourth equality.
    Note that, in particular, this shows
    \begin{equation}\label{eq:proof:linear_resolvent}
        T_\mu = T_\nu \circ \big( I + (\nu-\mu) T_\mu \big).
    \end{equation}

    Now consider the case that $\mu, \nu >\lambda_0$. Using Definition \ref{def:mild_sol_HJB}, and the shorthand $u(x) = R(\nu)(\psi)(x)$, we find that showing the resolvent identity in equation \eqref{eq:prop:id_of_resolvents} is equivalent to showing
    \[
        T^{\psi + (\mu-\nu) u}_\mu (u) = u.
    \]
    Then, we find
    \begin{align}
        T_\mu^{\psi+(\mu-\nu)u}(u)
        &= T_\mu\big[\psi + (\mu - \nu)u + \Hmin(\nabla^B u)\big] \\
        &= T_\mu\big[\psi + \Hmin(\nabla^B u) + (\mu - \nu) T_\nu[\psi + \Hmin(\nabla^B u)]\big] \\
        &= T_\nu\big[\psi + \Hmin (\nabla^B u)\big]\\
        &= u,
    \end{align}
    where we used the form of the solution $R(\nu)(\psi)$, cf.\ Definition \ref{def:mild_sol_HJB}, for the second and last equalities and the result in equation \eqref{eq:proof:linear_resolvent} for the third equality.

    \smallskip
    \noindent{\bfseries Proof of \ref{item:injective_resolvent}:}
    Let $\psi,\phi \in \LUC(\Hbar)$, $\mu > \lambda_0$, and assume that $R(\mu)(\psi) = R(\mu)(\phi)=u$. By the definition of $R(\mu)$ and Definition \ref{def:mild_sol_HJB}, we then have
    \[
        T_\mu\big[\psi + \Hmin (\nabla^B u)\big] = T_\mu\big[\phi + \Hmin (\nabla^B u)\big],
    \]
    which implies that $T_\mu\psi=T_\mu\phi$. Now using the identity of resolvent from Part \ref{item:id_of_resolvents}, we find 
    \[
        R(\nu)(\psi + (\nu-\mu)u) = R(\nu)(\phi + (\nu-\mu)u)
    \]
    for any $\nu>\mu>\lambda_0$. Analogously to above, this yields that $T_\nu\psi=T_\nu\phi$.

    Note that as $P_t$ is a $\cK$-continuous semigroup, cf.\ \cite[Proposition B.89]{FaGoSw17}, we have that, for any $x \in \overline{H}$, the map $t \mapsto P_t[\psi-\phi](x)$ is continuous.
    Now, using the injectivity of the Laplace transform, we find that $P_t[\psi](x) = P_t[\phi](x)$ for all $t>0$.
    
    Again using the $\cK$-continuity of $P_t$, for all $x \in \overline{H}$, it follows that
    \[
        \lim_{t \downarrow 0} P_t[\psi](x) = \psi(x) \quad\text{and}\quad \lim_{t \downarrow 0} P_t[\phi](x) = \phi(x),
    \]
    which, in turn, implies that $\psi(x) = \phi(x)$ for all $x \in \overline{H}$.

    \smallskip\noindent
    {\bfseries Proof of \ref{item:lipschitz_resolvent}:}
    Let $\phi, \psi \in \UCb$ and $\mu > \lambda_0$. First note that by Proposition \ref{prop:Tlambda_B_diff_UCb}, we have that $u= R(\mu)(\phi)$ and $v= R(\mu)(\psi)$ are elements of $\LUC^{1,B}(\Hbar)$. Now, let $\eps > 0$ and set $\lambda = \mu + \frac{1}{\eps}$ in the resolvent identity in equation \eqref{eq:prop:id_of_resolvents}. Then, we find
    \begin{align}
        u = R\bigg( \mu + \frac{1}{\eps} \bigg) \bigg( \phi + \frac{1}{\eps} u \bigg) 
        &= T_{\mu + \frac{1}{\eps}} \bigg[ \phi + \frac{1}{\eps} u + \Hmin \big(\nabla^B u\big) \bigg],\\
        v = R\bigg( \mu + \frac{1}{\eps} \bigg) \bigg( \phi + \frac{1}{\eps} v \bigg) 
        &= T_{\mu + \frac{1}{\eps}} \bigg[ \phi + \frac{1}{\eps} v + \Hmin \big(\nabla^B v\big) \bigg].
    \end{align}
    Now using Lemma \ref{lemma:nonlin_approx} and the notation from equation \eqref{eq:generator_pre_lim} below, we find
    \begin{equation}
        u-v = T_{\mu + \frac{1}{\eps}} \bigg[ \phi - \psi + \frac{1}{\eps} \big( N_\eps u - N_\eps v \big) - M_\eps\phi + M_\eps \psi  \bigg].
    \end{equation}
    The contractiveness of $N_t$, cf.\ equation \eqref{eq:approx_contractive}, implies that
    \begin{equation}
        \|u-v\| \leq \frac{1}{\mu + \frac{1}{\eps}} \Big( \|\phi - \psi\| + \frac{1}{\eps} \|u-v\| + \|M_\eps\phi\| + \|M_\eps \psi \| \Big).
    \end{equation}
    Rearranging now yields
    \[
        \mu\|u-v\| \leq \|\phi - \psi\| +  \|M_\eps\phi\| + \|M_\eps \psi \|, 
    \]
    from which the statement follows by letting $\eps \downarrow 0$.

\end{proof}

\begin{lemma}\label{lemma:nonlin_approx}
    There exists a contractive family of operators $(N_t)_{t\geq 0}$ approximating the dynamics generated by the minimum value Hamiltonian $\Hmin$, i.e.,
    \begin{enumerate}[(a)]
        \item For $u,v \in \LUC(\Hbar)$ and $t \geq 0$, we have
        \begin{equation}\label{eq:approx_contractive}
            \|N_t u - N_t v\| \leq \|u - v\|
        \end{equation}

        \item For $u \in \LUC^{1, B}(\Hbar)$ and
        \begin{equation}\label{eq:generator_pre_lim}
            M_t u \coloneqq \frac{N_t u - u}{t} - \Hmin(\nabla^B u),
        \end{equation}
        we have that $\lim_{t \downarrow 0} \|M_t u\| = 0$
    \end{enumerate}
\end{lemma}

Intuitively, we want $v(t)\colon t \mapsto N_t u$ to be the solution of the Cauchy problem
\begin{equation}
    \begin{cases}
        v'(t) = \Hmin (\nabla^B v(t)), \quad t>0,\\
        v(0) = u,
    \end{cases}
\end{equation}
where $\Hmin (\nabla^B v) = \inf_{k \in K} \Hcv (\nabla^B v; k)$. In other words, we expect $N_t$ to be a semigroup of Nisio-type w.r.t.\ the family of current value Hamiltonians. As usual in the literature, we prove the result using an auxiliary control problem. The proof below follows the strategy in \cite[Lemma 4.166]{FaGoSw17} with careful adjustments w.r.t.\ the notion of $\nabla^B$ and the underlying spaces.

\begin{proof}
    Consider the auxiliary drift control control problem
    \begin{equation}
        \begin{cases}
            y'(t) = \alpha(t) ,\quad t>0\\
            y(0)=x,
        \end{cases}
    \end{equation}
    where we denote the solution of the above equation as $y(s;\,x,\alpha)$, or simply $y(s)$ and the control $\alpha\colon [0, \infty) \to \overline{H}$ is taken from the set
    \[
        \Lambda_M = \Big\{ \alpha \in L^\infty(0,\infty;\, \overline{H}) \;\Big|\; \|\alpha(s)\|\leq M \text{ for a.e. } t \in [0, \infty) \Big\},
    \]
    where $M>0$ is a constant. Then, the cost functional, over which we aim to optimize our control, is given by
    \[
        J(t,x; \, \alpha) \coloneqq \int_0^t g(\alpha(s)) \dd s + u(y(t))
    \]
    with the running cost $g$ on $\overline{H}$ given by
    \[
        g(\alpha) = \sup_{\|p\| \leq M} \big\{\Hmin(p) - \ip{\alpha}{p} \big\}.
    \]
    Now, taking the infimum over all admissible controls, we find that the value function is
    \begin{equation}\label{eq:nonlin_approx:val_fct}
        v(t,x) = \inf_{\alpha \in \Lambda_M} J(t,x;\, \alpha)
    \end{equation}
    As stated in the heuristic explanation before, for $u \in \LUC(\Hbar)$ we now set $N_t u = v(t, \cdot)$ and proceed to show the claimed properties.

    \smallskip

    For the contractiveness of the operator $N_t$, let $u_1, u_2 \in \LUC(\Hbar)$. For $\eps > 0$ and $x\in \overline{H}$, we have
    \begin{align}
        \| N_{\eps} (u_1)(x) - N_{\eps} (u_2)(x) \| &\leq \sup_{\alpha \in \Lambda_M} \| u_1\big(y(\eps;\, x, \alpha)\big) - u_1\big(y(\eps;\, x, \alpha)\big)\|\\
        &\leq \|u_1 - u_2 \|.
    \end{align}

    \smallskip

    As the identification of the generator of $N_t$ is more involved, we show it in three steps:

    \noindent {\bfseries Step 1:}
    First we aim to show that, for $M \geq \lipn{\Hmin}$ and $p \in \overline{H}$ with $\|p\| \leq M$, we have
    \[
        \Hmin(p) = \inf_{\|\alpha\| \leq M} \big\{ \ip{\alpha}{p} + g(\alpha) \big\}.
    \]
    To that end, let $q \in \overline{H}$ with $\|q\| \leq M$ and set
    \[
        G(p) \coloneqq \inf_{\|\alpha\| \leq M} \big\{ \ip{\alpha}{p} + g(\alpha) \big\}.
    \]

    By the definition of $g$, we then find
    \[
        G(q) = \inf_{\|\alpha\| \leq M} \sup_{\|p\| \leq M} \big\{ \ip{\alpha}{q-p} + \Hmin(p) \big\} \geq \Hmin(q),
    \]
    where the inequality follows from choosing $p=q$.

    To show the other direction, first note that rewriting as above and then adding and subtracting $\Hmin(q)$, we have
    \[
        G(q) = \Hmin(q) + \inf_{\|\alpha\| \leq M} \sup_{\|p\| \leq M} \big\{ \ip{\alpha}{q-p} + \Hmin(p) - \Hmin(q) \big\}.
    \]
    As $\lipn{\Hmin} < \infty$, we know that the superdifferential of $\Hmin$ at $q$ is non-empty.
    Now using that as an infimum over affine functions $\Hmin$ is concave, we know that, for any $\xi$ in the superdifferential of $\Hmin$ at $q$, we have $\Hmin(p) - \Hmin(q) \leq \ip{\xi}{p-q}$ and consequently
    \[
        G(q) \leq \Hmin(q) + \inf_{\|\alpha\| \leq M} \sup_{\|p\| \leq M} \big\{ \ip{\alpha}{q-p} - \ip{\xi}{q-p}\big\}.
    \]
    This, in turn, implies that $G(q) \leq \Hmin(q)$ by choosing $\alpha = \xi$, which itself is possible because $\xi$ is an element of the superdifferential of $\Hmin$ and $M \geq \lipn{\Hmin}$.

    \noindent {\bfseries Step 2:}
    Now let $u \in \LUC^{1, B}(\Hbar)$ and $M\geq \max\big\{ \|\nabla^Bu\|, \lipn{\Hmin} \big\}$ and, for $\eps>0$, consider $M_t u$ as in equation \eqref{eq:generator_pre_lim}. Then, using the previous step and that the value function is bounded from above by the infimum over all constant controls, denoted by $a$, we find
    \begin{align}
        M_t (u)(x) &\leq \inf_{\|a\|\leq M}\bigg\{ g(a) + \frac{u(y(\eps;\, x, a)) - u(x)}{\eps} \bigg\} - \inf_{\|a\|\leq M}  \big\{ \ip{a}{\nabla^B u(x)} + g(a) \big\}\\
        &\leq \sup_{\|a\|\leq M}\bigg\{ \frac{u(y(\eps;\, x, a)) - u(x)}{\eps} - \ip{a}{\nabla^B u(x)} \bigg\}.
    \end{align}

    Due to $u \in \LUC^{1, B}(\Hbar)$ and the definition of $y(\eps;\, x, a)$, we have
    \begin{equation}\label{eq:proof:taylor_u}
        u(y(\eps;\, x, a)) - u(x) = \int_0^{\eps} \ip{a}{\nabla^B u (y(s;\, x, a))} \dd s,
    \end{equation}
    which, using the previous estimate, yields
    \[
        M_t (u)(x) \leq \sup_{\|a\|\leq M}\bigg\{ \frac{1}{\eps} \int_0^{\eps} \ip{a}{\nabla^B u (y(s;\, x, a))} \dd s - \ip{a}{\nabla^B u(x)} \bigg\}.
    \]
    Taking $\eps \downarrow 0$, using the uniform continuity of $\nabla^B u$, and the fact that $\|y(s;\, x, a)-x\| \to 0$ as $s \downarrow 0$, we find that the right-hand side of the equation above converges to $0$ uniformly.

    \noindent {\bfseries Step 3:}
    Staying in the setting of the previou step, we know that by the Dynamic Programming Principle for $y$, we can find an $\eps^2$-optimal control $\alpha^{\eps}$, i.e., there exists a $\alpha^{\eps} \in \Lambda_M$ such that
    \[
        N_{\eps} (u)(x) \geq - \eps^2 + \int_0^\eps g(\alpha^\eps(s)) \dd s + u(y^{\eps}(\eps)),
    \]
    where we use the shorthand $y^{\eps}(s) = y(s;\, x, \alpha^{\eps})$.
    Now, using an analogous identity to equation \eqref{eq:proof:taylor_u}, we arrive at
    \begin{align}
        M_t (u)(x) &\geq -\eps + \frac1\eps \int_0^\eps g(\alpha^\eps(s)) \dd s + \frac{u(y^{\eps}(\eps)) - u(x)}{\eps} - \Hmin(\nabla^B u (x))\\
        &= -\eps + \frac1\eps \int_0^\eps g(\alpha^\eps(s)) + \ip{\alpha^{\eps}(s)}{\nabla^B u (x)}  - \Hmin(\nabla^B u (x)) \dd s
    \end{align}
    Again using the uniform continuity of $\nabla^B u$, and the fact that $\|y(s;\, x, a)-x\| \to 0$ as $s \downarrow 0$, but now also Step 1 to find that the integrand in the last line is bounded from below by $ \Hmin(\nabla^B u (y^{\eps}(s)))- \Hmin(\nabla^B u (x))$, we find that the right-hand side converges to $0$ uniformly as $\eps \downarrow 0$.
\end{proof}

\subsection{Proof of Theorem \ref{thm:ext_all_lambda}}\label{sec:proof_main_thm}
This section contains the proof of our main theorem, Theorem \ref{thm:ext_all_lambda}. While it is structurally similar to \cite[Theorem 4.167]{FaGoSw17}, we need the theory developed throughout the previous sections for the objects to even be well defined. Furthermore, Propositions \ref{prop:Tlambda_B_diff_UCb} and \ref{lem:resolvent_properties} structurally play a key role within the proof.

\begin{proof}[Proof of Theorem \ref{thm:ext_all_lambda}]
    Recalling the constant $\lambda_0 >0$ and proof of \cite[Theorem 4.6]{BoGo25} and translating it to our notation, we have that, for any $\lambda \geq \lambda_0$, the solution $R(\lambda)(\ell_0)$ is a fixed point of $T_\lambda^{\ell_0}$.
    Now combining the results of Proposition \ref{lem:resolvent_properties}\ref{item:id_of_resolvents} and \ref{item:injective_resolvent} with Proposition \ref{prop:Tlambda_B_diff_UCb}, we have that $R(\lambda) \colon \LUC(\Hbar) \to \LUC^{1, B}(\Hbar)$ is injective and satisfies the resolvent identity. Consequently, we can apply \cite[Chapter VII.4, Theorem 1]{Yo65} 
    to find that there exists a unique linear operator $\cB \colon D(\cB) \subseteq \LUC^{1, B}(\Hbar) \to \LUC(\Hbar)$ such that
    \begin{equation}
        R(\lambda) = (\lambda - \cB)^{-1}
    \end{equation}
    for all $\lambda \geq \lambda_0$.
    Furthermore, using Proposition \ref{lem:resolvent_properties}\ref{item:lipschitz_resolvent} and the fact that $R(\lambda)$ is defined on the entirety of $\LUC(\Hbar)$, we know that $R(\lambda)$ is maximally dissipative and, using \cite[Chapter 1, Theorem 4.5]{Pa83}, that we can extended $R(\lambda)$ with the same properties to all $\lambda >0$.

    Thus, it remains to show that, for any $\lambda > 0$, we have $u = (\lambda - \cB)^{-1}(\psi)$ if and only if $u$ is a fixed point of $T^{\psi}_\lambda$, where $u \in \LUC^{1, B}(\Hbar)$ and $\psi \in \LUC(\Hbar)$.

    Recall that, for any $\lambda^* \geq \lambda_0$ and $\psi \in \LUC(\Hbar)$, $R(\lambda^*)(\psi) = (\lambda^* - \cB)^{-1} (\psi)$ is the unique fixed point of $T^{\psi}_\lambda$. Now let $\lambda>0$ and set $u = (\lambda - \cB)^{-1} (\psi)$. As the resolvent identity \eqref{eq:prop:id_of_resolvents} holds for any $\lambda > 0$, we know that
    \begin{equation}
        u  = ((\lambda + \lambda_0) - \cB)^{-1} (\psi + \lambda_0 u),
    \end{equation}
    which we know to be the unique fixed point of $T^{\psi + \lambda_0 u}_{\lambda + \lambda_0}$.
    As we have shown in equation \eqref{eq:proof:linear_resolvent} that the resolvent identity also holds for $T_\lambda$, we find
    \begin{align}
        u &= T_{\lambda + \lambda_0} \big[ \psi + \lambda_0 u + \Hmin(\nabla^B u) \big]\\
        &= T_{\lambda} \big[ \psi + \lambda_0 u + \Hmin(\nabla^B u) + \big(\lambda - (\lambda + \lambda_0)\big) u\big]\\
        &= T^\psi_\lambda u,
    \end{align}
    i.e., $u$ is a fixed point $T^\psi_\lambda$.
    Note that, since this argument also works in reverse, we have shown the statement.
\end{proof}

\section{Applications}\label{sec:applications}

\subsection{Controlled Stochastic Wave Equation}
\label{sec:wave_application}

This section relies on the framework developed in \cite[Section 6]{BoGo25}, to which we refer for a more in depth discussion. 

We analyze the vibrations of a membrane over a domain $\Omega \subset \mathbb{R}^d$ with fixed boundaries, driven by an internal control $f(t,x)$ and subject to stochastic perturbations. Then, the evolution of the state $u(t,x)$ is governed by the stochastic partial differential equation (SPDE)
\begin{equation}\label{eq:wave_spde}
\begin{cases}
    \partial_{tt} u(t,x) = c^2 \Delta u(t,x) + f(t,x) + \sigma \dd W(t,x) & \text{in } (0,\infty) \times \Omega, \\
    u(t,x) = 0 & \text{on } (0,\infty) \times \partial\Omega, \\
    u(0,x) = u_0(x), \quad \partial_t u(0,x) = v_0(x) & \text{in } \Omega.
\end{cases}
\end{equation}
Here, $c>0$ represents the propagation speed, the distributed control $f(t,x)$ takes values in the space $K = L^2(\Omega;\, \mathbb{C})$, and $W(t,x)$ is a cylindrical Wiener process on $L^2(\Omega;\, \mathbb{C})$. We consider the Hilbert space $H = H^1_0(\Omega;\mathbb{C}) \times L^2(\Omega;\mathbb{C})$ as our state space and endow it with the energy norm $\|(u,v)\|_H^2 = c^2\|\nabla u\|_{L^2}^2 + \|v\|_{L^2}^2$. By setting $X(t) = (u(t,\cdot), \partial_t u(t,\cdot))^T$, the SPDE is recast as an abstract first-order evolution equation on the Hilbert space $H$ with dynamics
\begin{equation}\label{eq:wave_abstract}
\dd X(t) = (A X(t) + B f(t))\dd t + G \dd W(t), \quad X(0) = (u_0, v_0)^T.
\end{equation}

The uncontrolled dynamics are governed by the unbounded operator $A$ with domain $D(A) = (H^2(\Omega;\, \mathbb{C}) \cap H_0^1(\Omega;\, \mathbb{C})) \times H_0^1(\Omega;\, \mathbb{C})$ defined as 
$A = \begin{pmatrix} 0 & I \\ c^2\Delta & 0 \end{pmatrix}$.
As is standard in the literature, cf.\ \cite[Chapter 3]{LaTr00}, $A$ generates a strongly continuous contraction semigroup $e^{tA}$ on $H$, and its eigenfunctions $\Phi_{n}\in H$ with corresponding eigenvalues $\mu_{n}\in \mathbb{C}$ constitute a Riesz basis for $H$.

The control operator $B\colon K\to H$ acts exclusively on the velocity component, i.e., $B$ is zero in the first component and the identity operator on $L^{2}(\Omega;\,\mathbb{C})$ in the second component. We restrict our control strategies to the space of progressively measurable processes taking values in a bounded, closed set $U$
\begin{equation*}\label{eq:admcontr2-wave}
\mathcal{U} \coloneqq \big\{ u\colon[0,\infty)\times \Omega \to U \subseteq K \;\big|\; u \text{ is progressively measurable} \big\}.
\end{equation*}

The noise is introduced via the operator $G\colon\mathbb{R}^{N}\to H$, given by $G =(0, \sigma)^T$, where $\sigma\colon \mathbb{R}^N \to L^2(\Omega;\mathbb{C})$ is a Hilbert--Schmidt operator. 

Our objective is to minimize the following discounted infinite-horizon cost functional over $\mathcal{U}$, given the initial state $x = (u_0, v_0)^T$
\begin{equation}\label{eq:wave_cost}
    J(x;\, f) = \mathbb{E} \left[ \int_0^\infty e^{-\lambda t} \left[ \ell_{0}(X(t)) + \ell_{1}(f(t))\right]dt \right].
\end{equation}

For the running cost on the state, we assume a finite-dimensional structure: $\ell_{0}(x)=\hat \ell_{0}(\cP x)$, where $\cP$ denotes the spectral projection onto the finite-dimensional subspace $V_N = \operatorname{span}\{\Phi_1, \dots, \Phi_N\}$. Moreover, we assume, that $\cP$ and $G$ share the identical image $V_{N}$. To apply our main result, Theorem \ref{thm:ext_all_lambda}, we have to verify that the PEL machinery from Section \ref{sec:struct_prop}, and in particular Hypothesis \ref{ipo-smoothing-lifting}, work in this setting.  
The following proposition guarantees that Hypothesis \ref{ipo-smoothing-lifting} is met.

\begin{proposition}
    Assume that the restriction of the operator $\cP GG^{*}\cP^{*}$ to $V_{N}$ is positive definite. Then, for every $t>0$, the abstract model satisfies the following bound:
    \begin{equation}
        \left\| (\Upsilon^{\cP}_\infty Q_t (\Upsilon^{\cP}_\infty)^*)^{-1/2} \Upsilon^{\cP}_\infty e^{tA}B \right\| \le C t^{-1/2},
    \end{equation}
    where $Q_t = \int_0^t e^{sA}GG^*e^{sA^*}\dd s$ is the covariance operator associated with the system.
\end{proposition}
\begin{proof}
By virtue of \cite[Theorem 6.3]{BoGo25}, an analogous estimate holds true when substituting $\Upsilon^{\cP}_\infty$ with the projection operator $\cP$. Furthermore, \cite[Lemma 6.2]{BoGo25} establishes that the semigroup $e^{tA}$ commutes with $\cP$. Combining this commutation property with the aforementioned weaker estimate directly yields the desired inequality, as detailed in \cite[Remark 3.26]{BoGo25}.
\end{proof}

\subsection{Stochastic Heat Equation with Boundary Control}\label{sec:application}

This section extends the results of \cite[Section 7]{BoGo25}. We refer the reader there for an exhaustive treatment of the topic in this context.

Here, we consider a optimal boundary control problem of a stochastic heat equation. Specifically, we will verify that Hypothesis \ref{ipo-smoothing-lifting} holds, and thus the PEL machinery works. As a consequence, we can then apply Theorem \ref{thm:ext_all_lambda}. For a broader perspective on boundary control problems, we refer to \cite{MoFa13}.

Let $\mathcal{O}\subseteq \mathbb{R}^d$ be an open, bounded, and connected domain with smooth boundary $\partial\mathcal{O}$. We study the evolution of a stochastic heat equation driven by a boundary control:
\begin{equation}\label{eqdiri}
\begin{cases}
    \partial_{t} y (s,\xi) = \Delta y(s,\xi)+\sigma \dd W(s,\xi), & s\in (0,\infty),\; \xi\in \mathcal{O}, \\
    y(0,\xi)=x(\xi),& \xi\in \mathcal{O},\\
    y(s,\xi)= u(s,\xi), & s\in (0,\infty),\; \xi\in \partial\mathcal{O},
\end{cases}
\end{equation}
where $W$ is a cylindrical Wiener process on $L^{2}(\mathcal{O})$, and the initial state and control satisfy $x(\cdot)\in L^{2}(\mathcal{O})$ and $u(s,\cdot)\in L^{2}(\partial \mathcal{O})$ for all $s > 0$, respectively.

We consider the state space $H=L^{2}(\mathcal{O})$. The uncontrolled dynamics are governed by the Dirichlet Laplacian $A = \Delta$ defined on the domain $D(A)=H^{2}(\mathcal{O})\cap H^1_0(\mathcal{O})$. As a self-adjoint operator, $A$ possesses a purely discrete, strictly negative point spectrum $\{-\lambda_n\}_{n\in \mathbb{N}}$, with the classical asymptotic behavior $\lambda_n\sim n^{2/d}$ as $n \to \infty$. Consequently, $H$ admits a complete orthonormal system of eigenfunctions $\{e_n\}_{n\in \mathbb{N}}$. It is a standard result \cite[Chapter 1 Theorem 4.3]{Pa83} that $A$ acts as the infinitesimal generator of an analytic semigroup $e^{tA}$ on $H$.

The control space is defined as $K=L^{2}(\partial \mathcal{O})$. We restrict the admissible control actions to the following set of progressively measurable processes
\begin{equation*}\label{eq:admcontr2-heat}
\mathcal{U}\coloneqq \big\{
u\colon[0,\infty)\times \Omega \to U \subseteq K \;\big|\; u \text{ is progressively measurable}
\big\},
\end{equation*}
where $U$ is a given closed and bounded subset of $K$.

To handle the boundary conditions, we construct an extrapolated state space $\overline{H}$ defined as the topological dual of $V\coloneqq D((-A)^{3/4+\varepsilon})$ for a sufficiently small $\varepsilon>0$. The conversion of boundary data into internal domain forcing relies on the Dirichlet map $\cD \colon L^{2}(\partial\mathcal{O})\rightarrow D((-A)^{1/4-\varepsilon})$, cf.\ \cite{LiMa72}, defined as the unique weak solution to the elliptic boundary value problem
\begin{equation}
    \begin{cases}
        \Delta f(\xi)=0, & \xi\in \mathcal{O}, \\
        f(\xi)=a(\xi), & \xi\in \partial\mathcal{O},
    \end{cases}
\end{equation}
where $a\in L^{2}(\partial\mathcal{O})$. The regularity property $\cD \in \mathcal{L}(L^{2}(\partial\mathcal{O});\, D((-A)^{1/4-\varepsilon}))$ follows from classical elliptic regularity theory, cf.\ \cite[Chapter 3 and Appendix A]{LaTr00}. 

The following definition comes out naturally from the fact that classical solutions to problem \ref{eqdiri} satisfy the following integral equation, see \cite[Proposition C.12]{FaGoSw17}.

\begin{definition}
An adapted process $X(\cdot)$ is a mild solution to \eqref{eqdiri} if the following integral equation holds for all $s\geq 0$:
\begin{equation}\label{mildiri}
X(s)=e^{sA}x-A\int_0^s e^{(s-r)A} \cD u(r)\dd r +\int_0^s e^{(s-r)A}G\dd W(r).
\end{equation}
\end{definition}

As established in, e.g., \cite[Section 7]{BoGo25}, the semigroup $\{e^{tA}\}_{t\geq0}$ admits a unique continuous extension $\{\overline{e^{tA}}\}_{t\geq0}$ on the larger space $\overline{H}$. 
By an analogous duality argument, the fractional operator $(-A)^{3/4+\epsilon}$ is extended to $\overline{(-A)^{3/4+\epsilon}}\in\mathcal{L}(H;\,\overline{H})$.

Setting $-\overline{A}\coloneqq \overline{(-A)^{3/4+\epsilon}}(-A)^{1/4-\epsilon}$, using the commutativity of the semigroup and the fractional power of its generator, as shown in \cite[Chapter 2, Theorem 6.13]{Pa83}, we can rewrite the boundary integral term as 
\begin{align}\label{conticino}
- A\int_0^s e^{(s-r)A}\cD u(r)\dd r &= \int_0^s \overline{(-A)^{3/4+\epsilon}}e^{(s-r)A}(-A)^{1/4-\epsilon} \cD u(r)\dd r \\
 &= \int_0^s e^{(s-r)A}(-\overline{A}) \cD u(r)\dd r.
\end{align}
Combining \eqref{mildiri} and \eqref{conticino} we can reformulate the original boundary control problem \eqref{eqdiri} into an abstract stochastic evolution equation on $\overline{H}$
\begin{equation}\label{eqdiri mild}
\begin{cases}
    \dd X(s)= AX(s)\dd s -\overline{A}\cD u(s)\dd s +G\dd W(s), \quad s\in(0,\infty) \\
    X(0)=x\in H.
\end{cases}
\end{equation}
In this formulation, the unbounded boundary action is replaced by the abstract control operator $B\coloneqq -\overline{A}\cD$, which is a bounded operator from $K$ into $\overline{H}$. By virtue of Definition \ref{mildiri}, the trajectory $X(t;\, x,u)$ is guaranteed to remain in $H$ for all $t\geq0$, given $x\in H$ and $u\in \mathcal{U}$.

Our aim is then to minimize an infinite-horizon expected discounted cost. For a given initial state $x \in H$, we seek an admissible strategy in $\mathcal{U}$ to minimize, i.e., we consider the value function
\begin{equation}\label{costoastratto}
V(x) = \inf_{u \in \cU} J(x;\, u)= \inf_{u \in \cU} \mathbb{E} \left[\int_0^\infty e^{-\lambda s} \left[\ell_0(X(s))+\ell_1(u(s))\right]\dd s\right].
\end{equation}
We model the noise covariance structure via the fractional operator $G=(-A)^{-\beta}$ for some parameter $\beta\ge 0$. The corresponding trace class covariance operator, needed for the verification of Hypotesis \ref{ipo-smoothing-lifting}, takes the form
\begin{equation}\label{cov-heat}
Q_t=\int_0^t (-A)^{-2\beta}e^{2sA}\dd s = (-A)^{-2\beta-1}(I-e^{2tA}).
\end{equation}

Additionally, the following assumptions are needed for the running costs.
\begin{hypothesis}\label{hp:BCcost}
We assume that
\begin{enumerate}[(i)]
    \item there exists a finite collection of linearly independent vectors $\{v_{1},\dots,v_{n}\}\subset D((-A)^{\alpha})$ with $\alpha>\frac{1}{4}+\epsilon$ and $\cP$ being the projection onto the space spanned by $\{v_{1},\dots,v_{n}\}$. Furthermore, the state cost $\ell_{0}\colon H\to\mathbb{R}$ is a continuous and bounded map $\bar\ell_{0}\in C_{b}(\mathrm{im}(\cP))$ concatenated with the projection $\cP$ such that
    \begin{equation}
    \ell_{0}(x)= \bar\ell_0 \circ \cP(x).
    \end{equation}
    
    \item The control cost $\ell_{1} \colon U\to\mathbb{R}$ is measurable and bounded from below.
\end{enumerate}
\end{hypothesis}
The specific structure of $\ell_0$ together with Proposition \ref{prop-adjoint} \ref{item:prop-adjoint} guarantees that it belongs to the class of liftable functions $\mathcal{S}^P_\infty(\overline{H})$. 
This property is needed for ensuring the partial smoothing results.

\begin{proposition}
Consider the Ornstein--Uhlenbeck semigroup $\{R_{t}\}_{t\geq0}$ associated with the uncontrolled dynamics of \eqref{eqdiri mild}. Given the abstract boundary operator $B = -\overline{A}\cD$ and a projection operator $\cP$ onto the subspace generated by $\{v_{1},\dots,v_{n}\}$, Hypothesis \ref{ipo-smoothing-lifting} holds for a suitable exponent $\gamma\in(\frac{3}{4},1)$.
\end{proposition}
\begin{proof}
The proof follows as a direct consequence of the results established in \cite[Proposition 7.6]{BoGo25}.
\end{proof}

In conclusion, the derivation above confirms that the abstract hypotheses are fully satisfied for the stochastic heat equation under boundary control and, consequently, we can apply Theorem \ref{thm:ext_all_lambda}.

\printbibliography

@book {Pa83,
    AUTHOR = {Pazy, A.},
     TITLE = {Semigroups of linear operators and applications to partial
              differential equations},
    SERIES = {Applied Mathematical Sciences},
    VOLUME = {44},
 PUBLISHER = {Springer-Verlag, New York},
      YEAR = {1983},
     PAGES = {viii+279},
      ISBN = {0-387-90845-5},
   MRCLASS = {47D05 (34Gxx 35Fxx 35Gxx 47H20)},
  MRNUMBER = {710486},
MRREVIEWER = {H.\ O.\ Fattorini},
       DOI = {10.1007/978-1-4612-5561-1},
       URL = {https://doi.org/10.1007/978-1-4612-5561-1},
}

@book{DPZa14,
	AUTHOR = {Da Prato, Giuseppe and Zabczyk, Jerzy},
	TITLE = {Stochastic equations in infinite dimensions},
	SERIES = {Encyclopedia of Mathematics and its Applications},
	VOLUME = {152},
	EDITION = {Second},
	PUBLISHER = {Cambridge University Press, Cambridge},
	YEAR = {2014},
	PAGES = {xviii+493},
	ISBN = {978-1-107-05584-1},
	DOI = {10.1017/CBO9781107295513},
	URL = {https://doi.org/10.1017/CBO9781107295513},
}

@book {FaGoSw17,
	AUTHOR = {Fabbri, Giorgio and Gozzi, Fausto and Swiech, Andrzej},
	TITLE = {Stochastic optimal control in infinite dimension},
	SERIES = {Probability Theory and Stochastic Modelling},
	VOLUME = {82},
	NOTE = {Dynamic programming and HJB equations,
	With a contribution by Marco Fuhrman and Gianmario Tessitore},
	PUBLISHER = {Springer, Cham},
	YEAR = {2017},
	PAGES = {xxiii+916},
	ISBN = {978-3-319-53067-3},
	DOI = {10.1007/978-3-319-53067-3},
	URL = {https://doi.org/10.1007/978-3-319-53067-3},
}

@article {GoMa23,
    AUTHOR = {Gozzi, Fausto and Masiero, Federica},
     TITLE = {Stochastic control problems with unbounded control operators:
              solutions through generalized derivatives},
   JOURNAL = {SIAM J. Control Optim.},
  FJOURNAL = {SIAM Journal on Control and Optimization},
    VOLUME = {61},
      YEAR = {2023},
    NUMBER = {2},
     PAGES = {586--619},
      ISSN = {0363-0129,1095-7138},
   MRCLASS = {93E20 (35R15 47D07 49L20 60H20 93C23)},
  MRNUMBER = {4577194},
       DOI = {10.1137/22M1474679},
       URL = {https://doi.org/10.1137/22M1474679},
}

@article {GoMa25,
    AUTHOR = {Gozzi, Fausto and Masiero, Federica},
     TITLE = {Lifting partial smoothing to solve {HJB} equations and
              stochastic control problems},
   JOURNAL = {SIAM J. Control Optim.},
  FJOURNAL = {SIAM Journal on Control and Optimization},
    VOLUME = {63},
      YEAR = {2025},
    NUMBER = {3},
     PAGES = {1515--1559},
      ISSN = {0363-0129,1095-7138},
   MRCLASS = {93E20 (35R15 47D07 49L20)},
  MRNUMBER = {4902119},
       DOI = {10.1137/23M1578450},
       URL = {https://doi.org/10.1137/23M1578450},
}

@article {GoMa17,
    AUTHOR = {Gozzi, Fausto and Masiero, Federica},
     TITLE = {Stochastic optimal control with delay in the control {I}:
              {S}olving the {HJB} equation through partial smoothing},
   JOURNAL = {SIAM J. Control Optim.},
  FJOURNAL = {SIAM Journal on Control and Optimization},
    VOLUME = {55},
      YEAR = {2017},
    NUMBER = {5},
     PAGES = {2981--3012},
      ISSN = {0363-0129,1095-7138},
   MRCLASS = {93E20 (35K55 47D07 49L20 60H20)},
  MRNUMBER = {3702860},
MRREVIEWER = {Annalisa\ Cesaroni},
       DOI = {10.1137/16M1070128},
       URL = {https://doi.org/10.1137/16M1070128},
}

@book {DPZa02,
    AUTHOR = {Da Prato, Giuseppe and Zabczyk, Jerzy},
     TITLE = {Second order partial differential equations in {H}ilbert
              spaces},
    SERIES = {London Mathematical Society Lecture Note Series},
    VOLUME = {293},
 PUBLISHER = {Cambridge University Press, Cambridge},
      YEAR = {2002},
     PAGES = {xvi+379},
      ISBN = {0-521-77729-1},
   MRCLASS = {47D06 (28C20 35R15 46G12 47D07 47N20 93C20)},
  MRNUMBER = {1985790},
MRREVIEWER = {Sandra\ Cerrai},
       DOI = {10.1017/CBO9780511543210},
       URL = {https://doi.org/10.1017/CBO9780511543210},
}

@book {LaTr00,
    AUTHOR = {Lasiecka, Irena and Triggiani, Roberto},
     TITLE = {Control theory for partial differential equations: continuous
              and approximation theories. {I}},
    SERIES = {Encyclopedia of Mathematics and its Applications},
    VOLUME = {74},
      NOTE = {Abstract parabolic systems},
 PUBLISHER = {Cambridge University Press, Cambridge},
      YEAR = {2000},
     PAGES = {xxii+644+I4},
      ISBN = {0-521-43408-4},
   MRCLASS = {93-02 (35B37 47D06 47N70 49N10 93C20 93C25)},
  MRNUMBER = {1745475},
MRREVIEWER = {Luciano\ Pandolfi},
}

@book {LiMa72,
    AUTHOR = {Lions, J.-L. and Magenes, E.},
     TITLE = {Non-homogeneous boundary value problems and applications.
              {V}ol. {I}},
    SERIES = {Die Grundlehren der mathematischen Wissenschaften},
    VOLUME = {Band 181},
      NOTE = {Translated from the French by P. Kenneth},
 PUBLISHER = {Springer-Verlag, New York-Heidelberg},
      YEAR = {1972},
     PAGES = {xvi+357},
   MRCLASS = {35JXX (35KXX 35LXX 46E35)},
  MRNUMBER = {350177},
}

@article{MoFa13,
    author = {Moura, Scott J. and Fathy, Hosam K.},
    title = {Optimal Boundary Control of Reaction–Diffusion Partial Differential Equations via Weak Variations},
    journal = {Journal of Dynamic Systems, Measurement, and Control},
    volume = {135},
    number = {3},
    pages = {no. 034501},
    year = {2013},
    month = {02},
    issn = {0022-0434},
    doi = {10.1115/1.4023071},
    url = {https://doi.org/10.1115/1.4023071}
}

@misc{BoGo25,
      title={Lifting and partial smoothing for stationary {HJB} equations and related control problems in infinite dimensions}, 
      author={Gabriele Bolli and Fausto Gozzi},
      year={2025},
      eprint={2510.25894},
      archivePrefix={arXiv},
      primaryClass={math.OC},
      url={https://arxiv.org/abs/2510.25894}, 
}

@book {barbu2010,
    AUTHOR = {Barbu, Viorel},
     TITLE = {Nonlinear differential equations of monotone types in {B}anach
              spaces},
    SERIES = {Springer Monographs in Mathematics},
 PUBLISHER = {Springer, New York},
      YEAR = {2010},
     PAGES = {x+272},
      ISBN = {978-1-4419-5541-8},
   MRCLASS = {34-02 (34G20 34G25 35A16 47J35 47N20)},
  MRNUMBER = {2582280},
MRREVIEWER = {Jean\ Mawhin},
       DOI = {10.1007/978-1-4419-5542-5},
       URL = {https://doi.org/10.1007/978-1-4419-5542-5},
}

@misc{bodefe26,
      title={Optimal control of stochastic Volterra integral equations with completely monotone kernels and stochastic differential equations on Hilbert spaces with unbounded control and diffusion operators}, 
      author={Gabriele Bolli and Filippo de Feo},
      year={2026},
      eprint={2602.17578},
      archivePrefix={arXiv},
      primaryClass={math.OC},
      url={https://arxiv.org/abs/2602.17578}, 
}

@article{MaTe22,
 author = {Masiero, Federica and Tessitore, Gianmario},
 title = {Partial smoothing of delay transition semigroups acting on special functions},
 fjournal = {Journal of Differential Equations},
 journal = {J. Differ. Equations},
 issn = {0022-0396},
 volume = {316},
 pages = {599--640},
 year = {2022},
 language = {English},
 doi = {10.1016/j.jde.2022.01.054},
 zbMATH = {7486725},
 Zbl = {1491.60087}
}

@article{CadP91,
 author = {Cannarsa, Piermarco and da Prato, Giuseppe},
 title = {Second-order {Hamilton}-{Jacobi} equations in infinite dimensions},
 fjournal = {SIAM Journal on Control and Optimization},
 journal = {SIAM J. Control Optim.},
 issn = {0363-0129},
 volume = {29},
 number = {2},
 pages = {474--492},
 year = {1991},
 language = {English},
 doi = {10.1137/0329026},
 url = {hdl.handle.net/11384/91967},
 zbMATH = {13523},
 Zbl = {0737.49020}
}

@incollection{CadP92,
 author = {Cannarsa, P. and Da Prato, G.},
 title = {Direct solution of a second order {Hamilton}-{Jacobi} equation in {Hilbert} spaces},
 booktitle = {Stochastic partial differential equations and applications. Proceedings of the third meeting on stochastic partial differential equations and applications held at Villa Madruzzo, Trento, Italy, January 1990},
 isbn = {0-582-10051-8},
 pages = {72--85},
 year = {1992},
 publisher = {Harlow: Longman Scientific \& Technical; New York: Wiley},
 language = {English},
 zbMATH = {503099},
 Zbl = {0805.49016}
}

@article{Go96,
 author = {Gozzi, Fausto},
 title = {Global regular solutions of second order {Hamilton}-{Jacobi} equations in {Hilbert} spaces with locally {Lipschitz} nonlinearities},
 fjournal = {Journal of Mathematical Analysis and Applications},
 journal = {J. Math. Anal. Appl.},
 issn = {0022-247X},
 volume = {198},
 number = {2},
 pages = {399--443},
 year = {1996},
 language = {English},
 doi = {10.1006/jmaa.1996.0090},
 zbMATH = {902509},
 Zbl = {0858.35129}
}

@article{ChMe97,
 author = {Chow, Pao-Liu and Menaldi, Jose-Luis},
 title = {Infinite-dimensional {Hamilton}-{Jacobi}-{Bellman} equations in {Gauss}-{Sobolev} spaces},
 fjournal = {Nonlinear Analysis. Theory, Methods \& Applications},
 journal = {Nonlinear Anal., Theory Methods Appl.},
 issn = {0362-546X},
 volume = {29},
 number = {4},
 pages = {415--426},
 year = {1997},
 language = {English},
 doi = {10.1016/S0362-546X(96)00043-0},
 url = {digitalcommons.wayne.edu/mathfrp/39},
 zbMATH = {1036478},
 Zbl = {0890.49012}
}

@article{GoGo06,
 author = {Goldys, B. and Gozzi, F.},
 title = {Second order parabolic {Hamilton}-{Jacobi}-{Bellman} equations in {Hilbert} spaces and stochastic control: {{\(L^{2}_{\mu}\)}} approach},
 fjournal = {Stochastic Processes and their Applications},
 journal = {Stochastic Processes Appl.},
 issn = {0304-4149},
 volume = {116},
 number = {12},
 pages = {1932--1963},
 year = {2006},
 language = {English},
 doi = {10.1016/j.spa.2006.05.006},
 zbMATH = {5083386},
 Zbl = {1111.49021}
}

@article{FuTe02,
 author = {Fuhrman, Marco and Tessitore, Gianmario},
 title = {Nonlinear {Kolmogorov} equations in infinite dimensional spaces: the backward stochastic differential equations approach and applications to optimal control},
 fjournal = {The Annals of Probability},
 journal = {Ann. Probab.},
 issn = {0091-1798},
 volume = {30},
 number = {3},
 pages = {1397--1465},
 year = {2002},
 language = {English},
 doi = {10.1214/aop/1029867132},
 zbMATH = {1906090},
 Zbl = {1017.60076}
}

@article{FuTe04,
 author = {Fuhrman, Marco and Tessitore, Gianmario},
 title = {Infinite horizon backward stochastic differential equations and elliptic equations in {Hilbert} spaces.},
 fjournal = {The Annals of Probability},
 journal = {Ann. Probab.},
 issn = {0091-1798},
 volume = {32},
 number = {1B},
 pages = {607--660},
 year = {2004},
 language = {English},
 doi = {10.1214/aop/1079021459},
 zbMATH = {2100729},
 Zbl = {1046.60061}
}

@article{FuTe2002,
 author = {Fuhrman, Marco and Tessitore, Gianmario},
 title = {The {Bismut}-{Elworthy} formula for backward {SDE}'s and applications to nonlinear {Kolmogorov} equations and control in infinite dimensional spaces},
 fjournal = {Stochastics and Stochastics Reports},
 journal = {Stochastics Stochastics Rep.},
 issn = {1045-1129},
 volume = {74},
 number = {1-2},
 pages = {429--464},
 year = {2002},
 language = {English},
 doi = {10.1080/104S1120290024856},
 zbMATH = {1886381},
 Zbl = {1013.60049}
}

@book {Yo65,
    AUTHOR = {Yosida, K\^osaku},
     TITLE = {Functional analysis},
    SERIES = {Die Grundlehren der mathematischen Wissenschaften},
    VOLUME = {Band 123},
 PUBLISHER = {Academic Press, Inc., New York; Springer-Verlag, Berlin},
      YEAR = {1965},
     PAGES = {xi+458},
   MRCLASS = {47.00 (46.00)},
  MRNUMBER = {180824},
MRREVIEWER = {Ivan\ Singer},
}
\end{document}